\documentclass{article}
\usepackage[T1]{fontenc}
\usepackage{hyperref}
\usepackage{mathrsfs} 
\usepackage{numprint}
\usepackage{amsthm}
\usepackage{amsmath} 
\usepackage{mathrsfs} 
\usepackage{amssymb} 
\usepackage{amsfonts} 
\usepackage{multicol}
\usepackage{pdfpages}
\usepackage{a4wide}

\usepackage{pythontex}
\usepackage{graphicx} 
\usepackage{wrapfig} 
\usepackage{stmaryrd}
\usepackage{enumitem} 
\usepackage{comment}

\usepackage{multicol}
\usepackage{tikz}
\usepackage{graphicx}
\usepackage{colortbl}
\usepackage{array}
\usepackage[absolute]{textpos} 
\usepackage{color} 
\definecolor{Prune}{RGB}{99,0,60}
\usepackage{mdframed}
\usepackage{fancyhdr}
\usepackage{geometry}
\usepackage{ragged2e}

\newcommand{\R}[0]{\mathbb{R}}

\newcommand{\N}[0]{\mathbb{N}}
\newcommand{\Z}[0]{\mathbb{Z}}
\newcommand{\Q}[0]{\mathbb{Q}}
\renewcommand{\P}[0]{\mathbb{P}}
\newcommand{\Nx}[0]{\mathbb{N}\setminus \{0 \}}

\newcommand{\RR}[0]{\mathcal{R}}

\newcommand{\Ind}[0]{\text{Ind}}

\newcommand{\DD}[0]{\mathcal{D}}

\renewcommand{\mod}[0]{\text{ mod }}

\newcommand{\II}[0]{\mathcal{I}}

\newcommand{\floor}[1]{\left\lfloor #1 \right\rfloor}

\let\oldforall\forall
\renewcommand{\forall}{\oldforall \, }

\newcommand{\tir}[0]{\text{-}}

\newcommand{\Span}[0]{\text{Span}}
\newcommand{\rank}[0]{\text{rank}}

\theoremstyle{definition}
\newtheorem{theo}{Theorem}[section]
\newtheorem{req}[theo]{Remark}

\newtheorem{lem}[theo]{Lemma}
\newtheorem{defi}[theo]{Definition}

\newcounter{constante}
\newcommand{\cons}[0]{\refstepcounter{constante}c_{\theconstante}}

\makeatletter
\renewcommand\theequation%
{\thesection.\arabic{equation}}
\@addtoreset{equation}{section}
\makeatother

\usepackage{dsfont}

\title{\textbf{Construction of subspaces with known Diophantine exponents for the last angle}}
\author{Gaétan GUILLOT}
\date{}

\begin{document}
\newgeometry{top=2.25 cm, bottom=2.25cm, left=2.6cm, right=2.6cm}
\maketitle
\begin{abstract}
 Schmidt generalized in 1967 the theory of classical Diophantine approximation to subspaces of $\R^n$. We consider Diophantine exponents for linear subspaces of $\R^n$ which generalize the irrationality measure for real numbers. Using geometry of numbers, we construct subspaces of $\R^n$ for which we are able to compute the associated exponents for the last angle.
\end{abstract}

\maketitle

\section{Introduction}

Classical Diophantine approximation deals with how closely real numbers (or points in $\R^n$) can be approximated by rational numbers (or rational points). In 1967, Schmidt \cite{Schmidt} proposed a broader version of this problem, focusing on the approximation of subspaces of $\R^n$ by rational subspaces. In this context, we will briefly outline the key concepts needed for this study, according to the definitions and notation from \cite{Schmidt}, \cite{Joseph-these}, \cite{joseph_exposants}, and \cite{joseph_spectre}. The result presented in this article comes from the Ph.D thesis of the author, see \cite[Chapter 9]{Guillot_these}.

\bigskip 

A subspace of $\R^n$ is rational if it admits a basis of vectors in $\Q^n$. The set of all rational subspaces of dimension $e$ in $\R^n$ is denoted by $\RR_n(e)$. To such a rational subspace $B$, we can associate a point 
 $\eta = (\eta_1, \ldots, \eta_N) \in \P^N(\R)$ with $N = \binom{n}{e}$ known as the Grassmann (or Plücker) coordinates of $B$. We can select a representative vector $\eta$ with coprime integer coordinates and we define
$ H(B) = \| \eta \|$
where $\| \cdot \|$ denotes the Euclidean norm on $\R^N$. Note that if $X_1, \ldots, X_e$ is a $\Z$-basis of $B \cap \Z^n$, then $H(B) = \| X_1 \wedge \ldots \wedge X_e \|$. Further details on the height can be found in \cite{Schmidt} and \cite{Schmidt_book}.
Let us fix $n \in \Nx$. For $d,e \in \llbracket 1, n \rrbracket^2$ such that $d+ e \leq n$ and $j \in \llbracket 1, 
 \min(d,e) \rrbracket $ we say that a subspace $A$ of dimension $d$ of $\R^n$ is $(e,j)\tir$irrational if 
$\forall B \in \RR_n(e), \quad \dim( A \cap B) < j.$
We denote by $\II_n(d,e)_j$ the set of all $(e,j)\tir$irrational subspaces $A$ of dimension $d$ of $\R^n$.

We now introduce the concept of proximity between two subspaces. Let $ X, Y \in \R^n \setminus \{ 0 \} $, and define
\[ \omega ( X,Y) = \frac{\| X \wedge Y \|}{ \|X \| \cdot\|Y \| }, \]
where $ X \wedge Y $ represents the exterior product of $X$ and $Y$. Geometrically, $ \omega(X,Y) $ denotes the absolute value of the sine of the angle between $X$ and $Y$. Consider two subspaces $A$ and $B$ of $\R^n$ with dimensions $d$ and $e$ respectively. As in \cite{Schmidt}, we construct by induction $t = \min(d,e)$ angles between $A$ and $B$. 
Let us define 
\[ \psi_1(A,B) = \min\limits_{\substack{X \in A \setminus \{ 0 \} \\ Y \in B \setminus \{ 0 \}}} \omega(X,Y), \]
and choose $(X_1, Y_1) \in A \times B$ be such that $ \omega(X_1, Y_1) = \psi_1(A,B) $. Assume that $ \psi_1(A,B), \ldots, \psi_j(A,B) $ and $ (X_1, Y_1), \ldots, (X_j, Y_j) $ have been constructed for $j \in \llbracket 1, t-1 \rrbracket$. Let $A_j$ and $B_j$ be respectively the orthogonal complements of $ \text{Span} ( X_1, \ldots, X_j) $ in $A$ and $ \text{Span} (Y_1, \ldots, Y_j) $ in $B$. 

We define 
\[ \psi_{j+1}(A,B) = \min\limits_{\substack{X \in A_j \setminus \{ 0 \} \\ Y \in B_j \setminus \{ 0 \}}} \omega(X,Y), \]
and let $(X_{j+1}, Y_{j+1}) \in A \times B$ such that $ \omega(X_{j+1}, Y_{j+1}) = \psi_{j+1}(A,B) $.

We now have all the tools to define the Diophantine exponents studied in this paper.
\begin{defi}
 Let $(d,e) \in \llbracket 1,n-1 \rrbracket^2$ be such that $d+e \leq n$ and $j \in \llbracket 1, \min(d,e) \rrbracket $, and $A \in \II_n(d,e)_j$. We define $\mu_n(A|e)_j$ as the supremum of the set of all $\mu > 0$ such that there exist infinitely many $B \in \RR_n(e)$ such that
 $$ \psi_j(A,B) \leq H(B)^{-\mu}.$$
\end{defi}

The angle corresponding with $j = \min(d,e)$ is the most natural to study. Indeed, $\psi_{\min(d, e)}$ is "almost" a distance in the sense:
\begin{align*}
\psi_{\min(d, e)}(A, B) = 0 \Longleftrightarrow A \subset B \text{ or } B \subset A.
\end{align*}
In particular, if $d = e$, then $\psi_d$ is a distance on the Grassmannian of $d$-dimensional vector subspaces of $\R^n$. In \cite{Saxce}, de Saxcé describes the image of $\II_n(d,e)_{\min(d,e)}$ by $\mu_n(\cdot|e)_{\min(d,e)}$. In this article we study the joint spectrum associated to the last angle, that is to say the image of the tuple of functions $(\mu_n(\cdot|1)_{\min(d,1)}, \ldots, \mu_n(\cdot|n-d)_{\min(d,n-d)})$ on $\bigcap\limits_{e =1}^{n-d} \II_n(d,e)_{\min(d,e)}$. We prove here the following theorem. 

\begin{theo}\label{8theo_dernier_angle}
Let $(d,q) \in (\Nx)^2$. We set $n = (q+1)d$. There exists an explicit constant $C_d$ such that for any $\alpha \geq C_d$, there exists a subspace $A$ of $\R^n$ of dimension $d$ such that $A \in \bigcap\limits_{e =1}^{n-d} \II_n(d,e)_{\min(d,e)}$ and
\begin{align*}
\forall e \in \llbracket d, n-d \rrbracket, \quad &\mu_n(A| e)_d = \frac{\alpha^{q_e + 1}}{r_e + (d-r_e)\alpha}, \\
\forall e \in \llbracket 1, d-1 \rrbracket, \quad &\mu_n(A| e)_e = \frac{\alpha}{r_e} = \frac{\alpha}{e},
\end{align*}
where $q_e$ and $r_e$ are the quotient and the remainder of the Euclidean division of $e$ by $d$.
\end{theo}

In the case $d = 1$, the image of $(\mu_n(\cdot|1)_{1}, \ldots, \mu_n(\cdot|n-1)_{1})$ has been determined by Roy \cite{Roy}. In the case $d >1$, the determination of the joint spectrum remains an open problem. In \cite{article_Guillot}, assuming that $d$ divides $n$, the author provides examples of points in this image (see Theorem~1.8) and, in particular, identifies a non-empty open subset within it (see the proof of Theorem~1.5). While in \cite{article_Guillot}, one explicitly identifies the rational subspaces of best approximations of a subspace $A$ to compute $\mu_n(A|e)_{\min(d,e)}$, here we employ tools of geometry of numbers (mainly Minkowski's theorem on convex bodies) to find specific vectors with integer coordinates in the best approximations. Theorem \ref{8theo_dernier_angle} provides us with values taken by the tuple of functions $(\mu_n(\cdot|1)_{\min(d,1)}, \ldots, \mu_n(\cdot|n-d)_{\min(d,n-d)})$ which we are able to compute using the method developed here.

\bigskip

The definition of $C_d$ and the construction of the subspace $A$ are carried out in Section~\ref{section_construction}. Section \ref{sect_good_approx} is devoted to the study of rational vectors and rational subspaces that shall give good approximations of $A$. To prove Theorem~\ref{8theo_dernier_angle}, we consider two cases depending on whether $e < d$ or $e \geq d$.
\\In both cases, the lower bound of the exponent $\mu_n(A|e)_{\min(d,e)}$ is shown by presenting a family of rational subspaces that approximate $A$ well (see Lemmas~\ref{8lem_psid_ACNe} and \ref{8lem_Angle_psie_DN}). For $e \geq d$, we get that the "best" subspaces $C$ approximating $A$ contain a certain rational subspace $B_{N+1,q_e}$ (Lemma~\ref{8lem_inclusion_BN_C}). We can then bound the height of the subspace $C$ from below (Lemma~\ref{8minoration_somme_espace}), and we conclude that $\mu_n(A|e)_d$ cannot be too large, thanks to Lemma~\ref{8lem_minoration_exposant}. For $e < d$, we show that the "best" subspaces $C$ approximating $A$ intersect non-trivially a certain rational subspace $D_{N,d}$ (Lemma~\ref{8lem_intersection_DN_C}). We achieve the upper bound of $\mu_n(A|e)_e$ in Lemma~\ref{8lem_minoration_exposant_e<d} by bounding $\psi_{1}(C \cap D_{N,d}, A)$ from below and thus, a fortiori, $\psi_e(C, A)$.

\bigskip

Lemma~\ref{8lem_minoration_exposant} and Lemma~\ref{8lem_intersection_DN_C} represent the most challenging aspects of the proof of Theorem \ref{8theo_dernier_angle}, and they crucially involve the geometry of numbers.
\section{Construction of the subspace \texorpdfstring{$A$}{}}\label{section_construction}

We define the constant $C_d$ as the smallest real number such that for any $\alpha \geq C_d$, we have: 
\begin{align}
 -\alpha^{2} + \alpha(2d+2) - d &\leq 0 \label{8lem_inegal_alpha1} \\
 - \frac{\alpha}{2} + d(d-1) +1 &\leq 0 \label{8lem_inegal_alpha2} \\
 -\alpha^2 +(1+2d)\alpha - d &\leq 0 \label{8lem_inegal_alpha3},
\end{align}
and for any $e \in \llbracket d, qd \rrbracket: $ 
\begin{align}
 dr_e -(d-r_e) \alpha \leq 0 \label{8lem_inegal_alpha4} \\
 \frac{\alpha^{q_e}}{d -r_e + \frac{1}{2}} -\frac{\alpha^{q_e+1}}{r_e +(d -r_e)\alpha} + 1 \leq 0\label{8lem_inegal_alpha5}
\end{align}
where $q_e$ and $r_e$ are the quotient and the remainder of the Euclidean division of $e$ by $d$. Elementary computations show that $2 < C_d \leq 3d(d+4)$, see \cite[Lemma 9.3]{Guillot_these}.

\bigskip
We fix $\alpha \geq C_d$. Recall that $n = (q+1)d$ and therefore $qd = n - d$. Here, we construct the subspace $A$ from Theorem~\ref{8theo_dernier_angle}. Let $\theta$ be a prime number greater than or equal to $5$. For $j \in \llbracket 1, d \rrbracket$, define $\phi_j : \N \to \llbracket 0, qd-1 \rrbracket$ by:
\begin{align*}
\phi_j(k) = k + (j-1)q \mod (qd)
\end{align*}
where $x \mod (qd)$ is the remainder of the Euclidean division of $x$ by $qd$.
\\Throughout this paper, we denote $\sigma_{i,j} = \sum\limits_{k = 0}^{+ \infty} \frac{u^{(i,j)}_k}{\theta^{ \lfloor \alpha^k \rfloor }}$ for $i \in \llbracket 0, qd-1 \rrbracket$ and $j \in \llbracket 1, d \rrbracket$, with sequences $u^{(i,j)}$ that we will chose using the following lemma. Using the fact that $\alpha \geq 2$, Roth's theoreme implies that $\sigma_{i,j}$ is transcendental, see \cite{Bugeaud_approx_cantor} for further details on the number constructed here.

\begin{lem}\label{lem_cont_suit}
There exist sequences $u^{(0,1)}, \ldots, u^{(qd-1,1)}, \ldots, u^{(0,d)}, \ldots, u^{(qd-1,d)}$ satisfying $\forall i \in \llbracket 0, qd-1 \rrbracket, \quad \forall j \in \llbracket 1, d \rrbracket, \quad \forall k \in \N$:
\begin{align}\label{8construc_suite_u}
u^{(i,j)}_k \left\{
\begin{array}{ll}
\in \{2, 3\} & \text{if } i = k + (j-1)q \mod (qd)\\
= 0 & \text{otherwise}
\end{array}
\right.
\end{align}
and such that the family $(\sigma_{0,1}, \ldots, \sigma_{qd-1,1}, \ldots, \sigma_{0,d}, \ldots, \sigma_{qd-1,d})$ is algebraically independent over $\Q$.
 
\end{lem}

\begin{proof}
 Let $\sigma_{0,1}, \ldots, \sigma_{qd-1,1}, \ldots, \sigma_{0,d}, \ldots, \sigma_{qd-1,d}$ be arbitrarily ordered and denoted by $\sigma_1, \ldots, \sigma_{d(n-d)}$, and let us reason by induction on $t \in \llbracket 1,d(n-d)\rrbracket$. We denote the sequences $u_k^{(i,j)}$ associated with $\sigma_{i,j}$ by $u_k^1, \ldots, u_k^{d(n-d)}$.

The set of algebraic numbers over $\Q$ is countable and the set of sequences $(u_k^1)$ satisfying \eqref{8construc_suite_u} is uncountable. Therefore, we choose a sequence such that $\sigma_1$ is transcendental over $\Q$. \\
Now, suppose that we have constructed $\sigma_1, \ldots, \sigma_t$ as an algebraically independent family over $\Q$ with $t \in\llbracket 1, d(n-d) -1 \rrbracket$. The set of algebraic numbers over $\Q(\RR,\sigma_1, \ldots, \sigma_t)$ is countable, but the set of sequences $(u_k^{t+1})_{k \in \N}$ satisfying \eqref{8construc_suite_u} is uncountable. Therefore, we can choose a sequence such that $\sigma_{t+1} $ is transcendental over $\Q(\sigma_0, \ldots, \sigma_t)$, completing the induction.

\end{proof}

From now on, we assume that the sequences $(u_k^{(i,j)})_{k \in \N} $ and $\sigma_{i,j}$ for $i \in \llbracket 0 , qd-1 \rrbracket $ and $j \in \llbracket 1,d \rrbracket$ satisfy the conclusion of Lemma~\ref{lem_cont_suit}. 
\begin{req}\label{8req_def_uk}
 For fixed $k$ and $j$, the integer $i = k + (j-1)q \mod (qd)$ is the unique integer in $\llbracket 0, qd-1 \rrbracket$ such that $u_k^{(i,j)} \neq 0$.
\end{req}

 Before defining the subspace $A$, we state a lemma in order to study more precisely the sequences $(u_k^{(i,j)})_{k \in \N} $.

\begin{lem}\label{8lem_unique_kj_pour_i}
Let $i \in \llbracket 0, qd-1 \rrbracket$ and $N \in \N$. There exists a unique pair $(k,j) \in \llbracket 0, q-1 \rrbracket \times \llbracket 1, d \rrbracket$ such that $u_{N+k}^{(i,j)} \neq 0$.
\end{lem}

\begin{proof}
\textbullet \, \underline{Uniqueness:} Suppose there exist $\ell_1, \ell_2 \in \llbracket 0, q-1 \rrbracket$ and $j_1, j_2 \in \llbracket 1, d \rrbracket$ such that $u_{N + \ell_1}^{(i,j_1)} \neq 0 \text{ and } u_{N + \ell_2}^{(i,j_2)} \neq 0.$
By the definition of $u_k^{(i,j)}$, we have:
\begin{align*}
N + \ell_1 +(j_1-1)q \equiv N + \ell_2 +(j_2-1)q \mod{(qd)}.
\end{align*}
By the uniqueness of Euclidean division by $q$, since $\ell_1, \ell_2 \in \llbracket 0, q-1 \rrbracket$, we have $\ell_1 = \ell_2$. Thus, $(j_1-1)q \equiv (j_2-1)q \mod{(qd)}$ and hence $(j_1 - j_2)q \equiv 0 \mod{(qd)}$. Since $j_1, j_2 \in \llbracket 1, d \rrbracket$, we have $j_1 = j_2$.

\textbullet \, \underline{Existence:} We write the Euclidean divisions of $i$ and $N$ by $q$, $i = qu + v \text{ and } N = qu' + v'$
with $v, v' \in \llbracket 0, q-1 \rrbracket$.\\
If $v \geq v'$, we set $k = v - v' \in \llbracket 0, q-1 \rrbracket$ and $j = (u - u' \mod d) + 1 \in \llbracket 1, d \rrbracket$. We then verify that $i = N + k + (j-1)q \mod{(qd)}$:
\begin{align*}
N + k + (j-1)q \mod{(qd)} &= qu' + v' + v - v' + q(u-u') \mod{(qd)} = qu + v \mod{(qd)} = i.
\end{align*}
If $v < v'$, we set $k = v - v' + q \in \llbracket 0, q-1 \rrbracket$ and $j = (u - u' - 1 \mod d) + 1 \in \llbracket 1, d \rrbracket$. We then verify that $i = N + k + (j-1)q \mod{(qd)}$:
\begin{align*}
N + k + (j-1)q \mod{(qd)} &= qu' + v' + v - v' + q + q(u-u'-1) \mod{(qd)} = qu + v \mod{(qd)} = i.
\end{align*}
\end{proof}

We define for $j \in \llbracket 1,d \rrbracket$, the vector $Y_j$ in $\R^n$ as:
\begin{align*}
 Y_j = \begin{pmatrix} 
 0 & \cdots&
 0 &
 1 &
 0 &
 \cdots &
 0 &
 \sigma_{0,j} &
 \cdots &
 \sigma_{qd-1,j}
 \end{pmatrix}^\intercal
\end{align*}
where the $j-$th coordinate is equal to $1$, the last $n-d$ are $\sigma_{0,j}, \ldots, \sigma_{qd-1,j}$ and the others are zero.
We then define the subspace $A$ from Theorem~\ref{8theo_dernier_angle} as $A =\Span(Y_1, \ldots, Y_d)$. By considering the first $d$ coefficients of the vectors $Y_j$, it is clear that $\dim(A) = d$.

\begin{lem}
Let $e \in \llbracket 1, n-d \rrbracket$, then the subspace $A$ is $(n-d,\min(d,e))\tir$irrational.
\end{lem}

\begin{proof}
 Let $B$ be a rational subspace of dimension $e$. Suppose by contradiction that $\dim(A \cap B) \geq \min(d,e) $. Let us distinguish between two cases, depending on whether $d \geq e$ or $d \leq e$.
 \\ \textbullet \, If $d \geq e$ then $A \cap B = B$. So there exists $X \in A \cap \Q^n \setminus \{0 \}$ of the form $X = \lambda_1 Y_1 + \ldots + \lambda_d Y_d$. By looking at the first $d$ coordinates of $X$ we have $(\lambda_1, \ldots, \lambda_d) \in \Q^n$. For $i \in \llbracket 0, qd-1 \rrbracket$, we have a $\Q$-linear relation between $1, \sigma_{i,1}, \ldots, \sigma_{i,d}$ by looking at the $(d+1+i)-$th coordinate of $X$ which leads to a contradiction.
 \\ \textbullet \, If $d \leq e$ then $A \cap B = A$. In particular, we then have $Y_1 \in B$. Denoting by $Z_1, \ldots, Z_{n-d}$ a rational basis of $B$, we have $Y_1\wedge Z_1 \wedge \ldots \wedge Z_{n-d} = 0$. This equality implies the nullity of any minor of size $n-d+1$ of the matrix $
\begin{pmatrix}
Y_1 & Z_1 & \cdots & Z_{n-d}
\end{pmatrix} \in \mathcal{M}_{n, n-d+1}(\R).$ Every such minor is polynomial with rational coefficients in the coefficients of $Y_1$, as the $Z_i$ are rational. We can also view each determinant as a polynomial in $\Q[X_0, \ldots, X_{n-d-1}]$ evaluated at the $\sigma_{i,1}$. Since these coefficients form an algebraically independent family over $\Q$, then each polynomial is identically zero.
Thus, we can replace the coefficients $\sigma_{i,1}$ of $Y_1$ by any real family, and the determinant will be zero. Using this, we shall show that any minor of size $n-d$ of the matrix $Q = \begin{pmatrix} Z_1 & \cdots & Z_{n-d} \end{pmatrix}$ vanishes. Let $\Delta$ be a submatrix of size $(n-d) \times (n-d)$ of $Q$, we denote by $\Ind(\Delta)$ the set of indices $1 \leq i_1 < \ldots < i_{n-d} \leq n $ of the rows of $Q$ from which $\Delta$ is extracted. We distinguish between two cases.

\begin{itemize}[label = $\diamond$]
\item If $1 \notin \Ind(\Delta)$, then we set $\sigma_{0, j} = \ldots = \sigma_{qd-1,0} = 0$ and we compute the minor of size $n-d+1$ of the matrix$ \begin{pmatrix}
Y_1 & Z_1 & \cdots & Z_{n-d}
\end{pmatrix} $ corresponding to the rows $1$ and $\Ind(\Delta)$. This minor is equal to $\pm \det(\Delta)$ and is zero, so $\det(\Delta)$ vanishes. 

\item If $ 1 \in \Ind(\Delta)$, let us fix $i \in \llbracket 0, qd-1 \rrbracket$ such that $d + i +1\notin \Ind(\Delta)$. We set $\sigma_{i,1} = 1$ and $\sigma_{k,1} =0 $ for $ k \in \llbracket 0, qd-1 \rrbracket \setminus \{i\}$. We compute the minor of size $n-d+1$ of the matrix$ \begin{pmatrix}
Y_1 & Z_1 & \cdots & Z_{n-d}
\end{pmatrix} $ corresponding to the rows $d+i+1$ and $\Ind(\Delta)$; it is equal to:
\begin{align*}
 \pm \det(\Delta) + \det(\Delta')
\end{align*}
where $\Delta'$ is a submatrix of $(n-d) \times (n-d)$ of $Q$ with $1 \notin \Ind(\Delta') = \Ind(\Delta) \setminus \{1\}$. Using the first case, we have $\det(\Delta') = 0$ and so $\det(\Delta) = 0$. 
\end{itemize}
We have thus shown that every minor of size $n-d$ of $Q$ vanishes. In particular, $\rank(Q) < n-d $ which is contradictory since $Z_1, \ldots, Z_{n-d}$ form a basis of $B$.

\end{proof}

\begin{req}
 We can actually prove that $A \in \II_n(d,n-d)_{1} = \bigcap\limits_{e =1}^{n-d} \II_n(d,e)_{1}$, see \cite[Lemma 9.6]{Guillot_these}.
\end{req}

\section{Rational subspaces of good approximation}\label{sect_good_approx}
In this section we define rational subspaces. We will show later that theses subspaces achieve good approximations of $A$.

For $i \in \llbracket 0, qd-1 \rrbracket$, $j \in \llbracket 1, d \rrbracket$, and $N \in \Nx$, we define the truncated sum: 
$\sigma_{i,j,N} = \sum\limits_{k=0}^{N} \frac{u_k^{(i,j)}}{\theta^{\lfloor \alpha^k \rfloor}} \in \frac{1}{\theta^{\lfloor \alpha^N \rfloor}} \Z$.
Now, for $j \in \llbracket 1, d \rrbracket$, we define the vector in $\Z^n$:\begin{align*}
 X_N^j = \theta^{\lfloor \alpha^N \rfloor} \begin{pmatrix} 
 0 & \cdots&
 0 &
 1 &
 0 &
 \cdots &
 0 &
 \sigma_{0,j,N} & \cdots & \sigma_{qd-1,j,N}
 \end{pmatrix}^\intercal
\end{align*}
where the $j-$th coordinate is equal to $1$, the last $n-d$ are $\sigma_{0,j,N}, \ldots, \sigma_{qd-1,j,N}$ and the others are zero.

Next, we denote for $N$ and $v$ two positive integers:
\begin{align}\label{8def_Bnv}
B_{N,v} = \Span(X_N^1, X_{N+1}^1, \ldots, X_{N+v-1}^1, X_N^2, \ldots, X_{N+v-1}^2, \ldots, X_N^d, \ldots, X_{N+v-1}^d)
\end{align}
which is a rational subspace by definition.
For $j \in \llbracket 1, d \rrbracket$, we note that:
\begin{align}\label{8rec_Xn_Un}
X_{N+1}^j = \theta^{\lfloor \alpha^{N+1} \rfloor - \lfloor \alpha^N \rfloor} X_N^j + U_{N+1}^j \text{ with } U_{N+1}^j = \begin{pmatrix}
0 & \cdots & 0 & u_{N+1}^{(0,j)}& \cdots & u_{N+1}^{(qd-1,j)}
\end{pmatrix}^\intercal.
\end{align}

We set $V_N^j = \frac{U_N^j}{\|U_N^j\|} \in \Z^n$ because the vectors $U_N^j$ have a unique non-zero coordinate according to the construction in \eqref{8construc_suite_u}. The vectors $V_N^j$ are thus vectors of the canonical basis of $\R^n$.
We also introduce the vectors:
\begin{align}\label{8vecteur_Zn}
Z_N^j = \frac{1}{\theta^{\lfloor \alpha^N \rfloor}} X_N^j
\end{align}
and we note that $Z_N^j \underset{N\to+\infty}{\longrightarrow} Y_j$ so we deduce that there exist constants $\cons \label{8cons_minor_norme_XN}$ and $\cons \label{8cons_major_norme_XN}$ independent of $N$ such that for any $N \in \N$:
\begin{align}\label{8norme_XN}
c_{\ref{8cons_minor_norme_XN}} \theta^{\alpha^N} \leq \|X_N^j\| \leq c_{\ref{8cons_major_norme_XN}} \theta^{\alpha^N}.
\end{align}
More precisely for $j \in \llbracket 1, d \rrbracket$, we have:
\begin{align}\label{8major_Yj_XNj}
\psi_1(\Span(Y_j), \Span(X_N^j)) = \omega(Y_j, Z_N^j) \leq \frac{\|Y_j - Z_N^j\|}{\|Y_j\|} \leq c_{\ref{8cons_major_angle_Yj_XNj}} \theta^{-\alpha^{N+1}}
\end{align}
with $\cons \label{8cons_major_angle_Yj_XNj} >0$ independent of $N$. 

\begin{lem}\label{8lem_dimBN_zbaseBN}
Let $v \in \llbracket 1, q \rrbracket$. Then the subspace $B_{N,v}$ has dimension $dv$. Moreover, the vectors $(X_N^j)_{j \in \llbracket 1,d\rrbracket} \cup (V_k^j)_{j \in \llbracket 1,d\rrbracket, k \in \llbracket N+1, N+v-1\rrbracket}$ form a $\Z$-basis of $B_{N,v} \cap \Z^n$.
\end{lem}

\begin{proof}
By induction on $v$ and using (\ref{8rec_Xn_Un}), we have:
\begin{align*}
B_{N,v} = \Span(X_N^1, V_{N+1}^1, \ldots, V_{N+v-1}^1, X_N^2, V_{N+1}^2, \ldots, V_{N+v-1}^2, \ldots, X_N^d, V_{N+1}^d, \ldots, V_{N+v-1}^d).
\end{align*}
Remark~\ref{8req_def_uk} allows us to assert that the $V_k^j$ considered here are all different. Furthermore, we recall that these are vectors of the canonical basis. We deduce in particular that $(X_N^j)_{j \in \llbracket 1,d\rrbracket} \cup (V_k^j)_{j \in \llbracket 1,d\rrbracket, k \in \llbracket N+1, N+v-1\rrbracket}$ form a free family since the $j-$th coefficient of $X_N^{j}$ is $\theta^{\floor{\alpha^N}}$ while the $j-$th coefficient of the other vectors of the family is zero for any $j \in \llbracket 1,d \rrbracket$. 
\\We first prove the lemma for $v = q$. Let $(a_{j,k})_{{j \in \llbracket 1,d\rrbracket, k \in \llbracket 0, q-1\rrbracket}} \in \left[0,1\right]^{qd}$ be such that:
\begin{align}\label{8vecteur_U}
U = \sum\limits_{j = 1}^d a_{j,0} X_N^j + \sum\limits_{k=1}^{q-1} \sum\limits_{j = 1}^d a_{j,k} V_{N+k}^j \in \Z^n
\end{align}
By examining the first $d$ coordinates of $U$, we find that $a_{j,0} \theta^{\lfloor \alpha^N \rfloor} \in \Z$ for all $j \in \llbracket 1,d \rrbracket$. We can write $a_{j,0} = \frac{x_j}{y_j}$ with $y_j \mid \theta^{\lfloor \alpha^N \rfloor}$ and $\gcd(x_j, \theta) = 1$ for all $j \in \llbracket 1,d \rrbracket$. Let $j_{\max} \in \llbracket 1,d \rrbracket$ be such that $\max\limits_{j = 1}^d y_j = y_{j_{\max}}$.
\\Now consider the integer $i = N + q(j_{\max} -1) \mod(qd) \in \llbracket 0, qd-1 \rrbracket$. By the definition of $u_k^{(i,j)}$ in (\ref{8construc_suite_u}), we have $u^{(i,j_{\max})}_N \neq 0$ and therefore $\forall k \in \llbracket 1, q-1\rrbracket, \quad \forall j \in \llbracket 1,d \rrbracket, \quad u_{N+k}^{(i,j)} = 0 $ according to Lemma~\ref{8lem_unique_kj_pour_i}.
By the definition of $U_{N+k}$ in (\ref{8rec_Xn_Un}) and $V_{N+k} = \frac{U_{N+k}}{\| U_{N+k}\|}$, the $(d+i)$-th coordinate of the vector 
$\sum\limits_{k=1}^{q-1} \sum\limits_{j = 1}^d a_{j,k} V_{N+k}^j$ is therefore zero. Hence, according to (\ref{8vecteur_U}) $\sum\limits_{j = 1}^d a_{j,0} \theta^{\lfloor \alpha^N \rfloor} \sigma_{i,j,N}= \sum\limits_{j = 1}^d \frac{x_j}{y_j} \theta^{\lfloor \alpha^N \rfloor} \sigma_{i,j,N}\in \Z.$
Since the $y_j$ are powers of $\theta$ and $\max\limits_{j = 1}^d y_j = y_{j_{\max}}$, we have $\frac{y_{j_{\max}}}{y_j} \in \Z$ and:
\begin{align}\label{8divisibilite}
\sum\limits_{j = 1}^d x_j \frac{y_{j_{\max}}}{y_j} \theta^{\lfloor \alpha^N \rfloor} \sigma_{i,j,N}\in y_{j_{\max}} \Z.
\end{align}
Now we notice that for $j \neq j_{\max}$, we have $i \not\equiv N +1 +q(j-1) \pmod{qd}$ and thus $u_N^{(i,j)} = 0$ in this case.
Since $\lfloor \alpha^{N-1} \rfloor < \lfloor \alpha^N \rfloor$, we have for all $j \neq j_{\max}$, $\theta \mid \theta^{\lfloor \alpha^N \rfloor} \sigma_{i,j,N}= \theta^{\lfloor \alpha^N \rfloor} \sum\limits_{k = 0}^{N} \frac{u_k^{(i,j)}}{\theta^{\lfloor \alpha^k \rfloor}}.$
If $\theta \mid y_{j_{\max}}$, then using (\ref{8divisibilite}) $\theta \mid x_{j_{\max}} \theta^{\lfloor \alpha^N \rfloor} \sigma_{i,j_{\max},N}.$
Since $u_n^{(i,j_{\max})}$ is non-zero and coprime with $\theta$, we have $\gcd(\theta^{\lfloor \alpha^N \rfloor} \sigma_{i,j_{\max},N}, \theta) = 1$.
\\We finally deduce that $\theta \mid x_{j_{\max}}$, which is contradictory to $\gcd(x_{j_{\max}}, \theta) = 1$. All $y_j$ are therefore equal to 1, and thus all $a_{j,0}$ are integers. Returning to (\ref{8vecteur_U}), we find
$\sum\limits_{k=1}^{q-1} \sum\limits_{j = 1}^d a_{j,k} V_{N+k}^j \in \Z^n.$
Since all the vectors $V_{N+k}^j$ are distinct and come from the canonical basis, $a_{j,k}$ is an integer for all $j \in \llbracket 1, d \rrbracket$ and $k \in \llbracket 0, q-1 \rrbracket$. This shows that $(X_N^j)_{j \in \llbracket 1,d\rrbracket} \cup (V_k^j)_{j \in \llbracket 1,d\rrbracket, k \in \llbracket N+1, N+q-1\rrbracket}$ forms a $\Z$-basis of $B_{N,q} \cap \Z^n$, and in particular that $\dim(B_{N,q}) = qd$.

\bigskip
Now let $v \in \llbracket 1,q - 1 \rrbracket$. We have $(X_N^j)_{j \in \llbracket 1,d\rrbracket} \cup (V_k^j)_{j \in \llbracket 1,d\rrbracket, k \in \llbracket N+1, N+v-1\rrbracket} \subset (X_N^j)_{j \in \llbracket 1,d\rrbracket} \cup (V_k^j)_{j \in \llbracket 1,d\rrbracket, k \in \llbracket N+1, N+q-1\rrbracket}$. Since it is contained in a $\Z-$basis, the family $(X_N^j)_{j \in \llbracket 1,d\rrbracket} \cup (V_k^j)_{j \in \llbracket 1,d\rrbracket, k \in \llbracket N+1, N+v-1\rrbracket}$ forms a $\Z$-basis of the $\Z$-module it generates. This $\Z$-module is $B_{N,v} \cap \Z^n$ and the lemma is proved.

\end{proof}

Using this lemma, we can compute the height of some specific rational subspaces in Lemmas \ref{8lem_hauteur_CN}, \ref{8minoration_somme_espace}, \ref{8lem_minoration_exposant} and \ref{8lem_hauteur_DN}.

\section{Computation of the exponent in the case \texorpdfstring{$e \geq d$}{ } }
In this section, we consider $e \in \llbracket d, qd \rrbracket$. Recall that $e = q_e d + r_e$ is the Euclidean division of $e$ by $d$. In particular, we have $1 \leq q_e \leq q$. The goal of this section is to compute $\mu_n(A| e)_d$.
\subsection{Lower bound on the exponent}\label{sect_4otherequality}
In this section, we introduce a sequence of rational subspaces of dimension $e$ that approximate $A$ well, which allows us to bound $\mu_n(A|e)_d$ from below. 
Let $N \in \N$. We define the subspace $C_{N,e}$ by:
\begin{align}\label{8def_Cn}
 C_{N,e} &= \Span(X_{N+1}^1, \ldots, X_{N+q_e}^1, \ldots, X_{N+1}^d, \ldots, X_{N+q_e}^d) \bigoplus \Span(X_N^1, \ldots, X_N^{r_e}) \\
 &= B_{N+1,q_e} \bigoplus \Span(X_N^1, \ldots, X_N^{r_e}) \nonumber
\end{align}
which is a rational subspace.
Using (\ref{8rec_Xn_Un}) and reasoning by induction for each $j \in \llbracket 1,d \rrbracket$, we have:
\begin{align}\label{8CN_rec}
 C_{N,e} = \, &\Span(X_{N}^1, V_{N+1}^1, \ldots, V_{N+q_e}^1, \ldots, X_{N}^{r_e}, V_{N+1}^{r_e}, \ldots, V_{N+q_e}^{r_e}) \nonumber\\ 
 &\bigoplus \Span(X_{N+1}^{r_e+1}, V_{N+2}^{r_e+1}, \ldots, V_{N+q_e}^{r_e+1} \ldots, X_{N+1}^{d}, V_{N+2}^{d}, \ldots, V_{N+q_e}^{d}).
\end{align}

\begin{req}\label{8req_CN=BN}
 Note that in the case where $r_e = 0$, by the definition in (\ref{8def_Cn}), we have $C_{N,e} = B_{N+1,q_e}$.
\\ In all cases, we have $B_{N+1,q_e} \subset C_{N,e} \subset B_{N,q_e+1}$.
\end{req} 

\begin{lem}\label{8lem_hauteur_CN}
We have $\dim(C_{N,e}) = e$. Moreover, there exist constants $\cons \label{8cons_haut_CN_minor} > 0$ and $\cons \label{8cons_haut_CN_major} > 0$ independent of $N$ such that 
\begin{align*}
 c_{\ref{8cons_haut_CN_minor}} \theta^{ r_e {\alpha^N} + (d-r_e) {\alpha^{N+1}}} \leq H(C_{N,e}) \leq c_{\ref{8cons_haut_CN_major}} \theta^{ r_e {\alpha^N} + (d-r_e) {\alpha^{N+1}}}.
\end{align*}
\end{lem}

\begin{proof}
If $q_e = q$, then $C_{N,e} = B_{N+1,q}$ because $e = qd$ and $r_e = 0$; otherwise $q_e < q$ and in this case $C_{N,e} \subset B_{N,q}$.
In each case, Lemma~\ref{8lem_dimBN_zbaseBN} states that the vectors considered in (\ref{8CN_rec}) come from a $\Z$-basis of $B_{N+1, q} \cap \Z^n$ or $B_{N,q} \cap \Z^n$ respectively. 
The relation (\ref{8CN_rec}) gives directly $\dim(C_{N,e}) = (q_e+1)r_e + (d-r_e)q_e = q_ed + r_e =e $. 

These vectors thus form a $\Z-$basis of $C_{N,e} \cap \Z^n$. Taking up the notation $ Z_{N}^j = \frac{1}{\theta^{\floor{\alpha^N}}} X_N^j $, we have:
\begin{align}\label{8Haut_produit}
 H(C_{N,e}) = &\theta^{r_e\floor{\alpha^N} +(d-r_e)\floor{\alpha^{N+1}}} \| H_N \| \leq \theta^{r_e{\alpha^N} +(d-r_e){\alpha^{N+1}}} \| H_N \| 
\end{align}
where $H_N$ is the exterior product of the vectors 
\begin{align}\label{8vecteurs_dans_haut_produit}
 (Z_N^j)_{j \in \llbracket 1,r_e \rrbracket} \cup (Z_{N+1}^j)_{j \in \llbracket r_e +1, d \rrbracket} \cup (V_{N+k}^j)_{j \in \llbracket 1, r_e \rrbracket, k\in \llbracket 1,q_e\rrbracket} \cup (V_{N+k}^j)_{j \in \llbracket r_e +1, d \rrbracket, k\in \llbracket 2,q_e\rrbracket}.
\end{align}
We can bound this norm from above by $\|H_N\| \leq \| Z_N^1 \wedge \ldots \wedge Z_N^{r_e} \wedge Z_{N+1}^{r_e+1} \wedge \ldots \wedge Z_{N+1}^{d} \| $
since the norms of the vectors $V_k^j$ are equal to $1$. 
Now the quantity $\| Z_N^1 \wedge \ldots \wedge Z_N^{r_e} \wedge Z_{N+1}^{r_e+1} \wedge \ldots \wedge Z_{N+1}^{d} \| $ converges to $ \| Y_1 \wedge \ldots \wedge Y_d \|$ as $N$ tends to infinity and is therefore bounded independently of $N$. There exists $ c_{\ref{8cons_haut_CN_major}}> 0 $, independent of $N$, such that:
\begin{align}\label{8major_HN}
 \|H_N \| \leq c_{\ref{8cons_haut_CN_major}}.
\end{align}

Furthermore, let us define the matrix $M$ whose column vectors are the vectors of $(\ref{8vecteurs_dans_haut_produit})$. Then, by taking up the construction of the vectors $V_N^j$ in $(\ref{8rec_Xn_Un})$, $M$ takes the form 
$ M = \begin{pmatrix}
 \begin{matrix}
 I_d \\
 \Sigma_N
 \end{matrix} &
 \begin{matrix}
 0 \\
 V_N
 \end{matrix}
 \end{pmatrix}$
where $\Sigma_N$ is a matrix whose coefficients are $\sigma_{i,j,N}$ or $\sigma_{i,j,N+1}$, and $V_N$ is a matrix of $\mathcal{M}_{qd, e -d }(\Z)$ of rank $e-d$ since its columns are $e-d$ distinct vectors from the canonical basis. 
\bigskip \\
Let $\Delta$ be a non-zero minor of $V_N$ of size $e-d$. 
We can then extract a square matrix $M'$ of size $e$ from $M$ by selecting the first $d$ rows and $e-d$ among the last ones, corresponding to the minor $\Delta$. The determinant of $M'$ is an integer because it is the product of $\det(I_d)= 1$ and $\Delta$ which is a minor of a matrix in $M_{e-d}(\Z)$. Hence, we have $|\det(M') |\geq 1$. Now $\det(M')$ is a minor of size $e$ of $M$, so we have $\|H_N\| \geq |\det(M')| \geq 1$. 
 By combining this with (\ref{8Haut_produit}) and (\ref{8major_HN}), we obtain:
 \begin{align*}
 c_{\ref{8cons_haut_CN_minor}}\theta^{r_e{\alpha^N} +(d-r_e){\alpha^{N+1}}} \leq \theta^{r_e\floor{\alpha^N} +(d-r_e)\floor{\alpha^{N+1}}} \leq H(C_{N,e}) \leq c_{\ref{8cons_haut_CN_major}}\theta^{r_e{\alpha^N} +(d-r_e){\alpha^{N+1}}} 
\end{align*}
by setting $ c_{\ref{8cons_haut_CN_minor}} = \theta^{-d}$.

\end{proof}

We now focus on the angle $\psi_d(A,C_{N,e})$. To do this, we first study the angle $\psi_1(\Span(Y_1), C_{N,e})$.

\begin{lem}\label{8lem_minoration_Y1angleCNE}
There exists a constant $\cons \label{8cons_minor_angle_y1_CNe} >0$ independent of $N$ such that:
\begin{align*}
\psi_1(\Span(Y_1), C_{N,e}) \geq c_{\ref{8cons_minor_angle_y1_CNe}} \theta^{- \alpha^{N+q_e+1}}.
\end{align*}
\end{lem}

\begin{proof}
Let $X \in C_{N,e} \setminus \{ 0\}$ such that $\psi_1(\Span(Y_1), C_{N,e}) = \omega(Y_1, X)$. We use the basis of $C_{N,e}$ as explicitly detailed in \eqref{8CN_rec}, and we denote $(a_{k,j})$ a family of real numbers such that $X$ can be written as:
\begin{align*}
\sum\limits_{j = 1}^{r_e} a_{0, j} \theta^{-\floor{\alpha^N}} X_N^j + \sum\limits_{j = r_e +1}^{d} a_{0, j} \theta^{-\floor{\alpha^{N+1}}}X_{N+1}^j + \sum\limits_{j = 1}^{r_e} \sum\limits_{k = 1}^{q_e} a_{k,j}V_{N+k}^j + \sum\limits_{j = r_e +1}^{d} \sum\limits_{k =1}^{q_e-1} a_{k,j}V_{N+k+1}^j.
\end{align*}
Without loss of generality, we assume that:
\begin{align}\label{8proof_aik_carre1}
\sum\limits_{j = 1}^{r_e} a_{0, j}^2+ \sum\limits_{j = r_e +1}^{d} a_{0, j} ^2 + \sum\limits_{j = 1}^{r_e} \sum\limits_{k = 1}^{q_e} a_{k,j}^2 + \sum\limits_{j = r_e +1}^{d} \sum\limits_{k =1}^{q_e-1} a_{k,j}^2 =1.
\end{align}
The norms of the vectors $\theta^{-\floor{\alpha^N}}X_{N}^j$ and $V_{N+k}^j$ are bounded by a constant independent of $N$, and we have $\omega(Y_1, X) = \frac{ \| Y_1 \wedge X \|}{ \| Y_1 \| \cdot \|X \|}$. Therefore, it suffices to show that if $N$ is large enough:
\begin{align}\label{8minor_angle_y1_CNe_proof}
\| Y_1 \wedge X \| \geq c_{\ref{8cons_minor_angle_y1_CNe}} \theta^{-\floor{ \alpha^{N+q_e+1}}} \geq c_{\ref{8cons_minor_angle_y1_CNe}}\theta^{-{ \alpha^{N+q_e+1}}}.
\end{align}

Recall that $Y_1 = \begin{pmatrix} 1 & 0 & \cdots & 0 & \sigma_{0,1} & \cdots & \sigma_{qd-1, 1} \end{pmatrix}^\intercal $ and let us explicit $X$ as:
\begin{align*}
\begin{pmatrix}
a_{0,1} \\
\vdots \\ 
a_{0,d} \\
\sum\limits_{j = 1}^{r_e} a_{0, j} \sigma_{0,j,N}+ \sum\limits_{j = r_e +1}^{d} a_{0, j} \sigma_{0,j,N+1} + \sum\limits_{j = 1}^{r_e} \sum\limits_{k = 1}^{q_e} a_{k,j}v_{N+k,0}^j + \sum\limits_{j = r_e +1}^{d} \sum\limits_{k =1}^{q_e-1} a_{k,j}v_{N+k+1,0}^j \\
\vdots \\
\sum\limits_{j = 1}^{r_e} a_{0, j} \sigma_{qd-1,j,N}+ \sum\limits_{j = r_e +1}^{d} a_{0, j} \sigma_{qd-1,j,N+1} + \sum\limits_{j = 1}^{r_e} \sum\limits_{k = 1}^{q_e} a_{k,j}v_{N+k,qd-1}^j + \sum\limits_{j = r_e +1}^{d} \sum\limits_{k =1}^{q_e-1} a_{k,j}v_{N+k+1,qd-1}^j
\end{pmatrix}
\end{align*}
with $V_{N+k}^j = \begin{pmatrix} 0 & \cdots & 0 &v_{N+k,0}^j & \cdots & v_{N+k,qd-1}^j \end{pmatrix}^\intercal$.
We then prove \eqref{8minor_angle_y1_CNe_proof} by considering different cases. Let $\sigma \geq 1 $ be an upper bound on the $\sigma_{i,j}$, especially for all $N \in \N$, $\sigma \geq \sigma_{i,j,N}$. 

\textbullet \: \underline{First case:} If there exists $j \in \llbracket 2, d \rrbracket$ such that
$|a_{0,j}| \geq \frac{\theta^{-\floor{ \alpha^{N+q_e+1}}}}{(\sigma qd)^2}$
then, by bounding $\| X \wedge Y_1\| $ from below by the minor of $(Y_1| X)$ corresponding to the first row and the $j$-th row, we find
$ \| X \wedge Y_1 \| \geq \left| \det \begin{pmatrix}
 1 & a_{0,1} \\
 0 & a_{0,j}
 \end{pmatrix} \right| \geq |a_{0,j}| \geq \frac{\theta^{-\floor{ \alpha^{N+q_e+1}}}}{(\sigma qd)^2}$
which yields $(\ref{8minor_angle_y1_CNe_proof})$.
\\ \textbullet \: \underline{Second case:} Otherwise $ \forall j \in \llbracket 2, d \rrbracket, \quad |a_{0,j}| < \frac{\theta^{-\floor{ \alpha^{N+q_e+1}}}}{(\sigma qd)^2}.$
According to $(\ref{8proof_aik_carre1})$, we have:
\begin{align*}
 a_{0, 1}^2 + \sum\limits_{j = 1}^{r_e} \sum\limits_{k = 1}^{q_e} a_{k,j}^2 + \sum\limits_{j = r_e +1}^{d} \sum\limits_{k =1}^{q_e-1} a_{k,j}^2 \geq 1 -(d-1) \left(\frac{\theta^{-\floor{ \alpha^{N+q_e+1}}}}{(\sigma qd)^2} \right)^2.
\end{align*}
In particular, if $N$ is large enough, there exists $(j', k')$ such that $|a_{k',j'}| \geq \frac{1}{qd}$ with 
 $k' > 0 \text{ or } (k'= 0 \text{ and } j' =1).$ First, let us consider the case $k' >0$. We then set $i = N+ k' + (j'-1)q \mod(qd)$. By definition of $u_{N+k}^{(i,j)}$ (and thus of $v_{N+k,i}^j$) in $(\ref{8construc_suite_u})$, we have:
\begin{align*}
 v_{N+k',i}^{j'} = 1 \text{ and } \forall (j,k) \neq (j',k'), \quad v_{N+k,i}^j = 0.
\end{align*}
By bounding $\| X \wedge Y_1\| $ from below by the minor of $(Y_1| X)$ corresponding to the first row and the $(i+1+d)$-th row, we find:
\begin{align*}
 \| X \wedge Y_1 \| &\geq \left| \det \begin{pmatrix}
 1 & a_{0,1} \\
 \sigma_{i,1} & \sum\limits_{j = 1}^{r_e} a_{0, j} \sigma_{i,j,N}+ \sum\limits_{j = r_e +1}^{d} a_{0, j} \sigma_{i,j,N+1} + a_{k',j'}v_{N+k',i}^{j'} 
 \end{pmatrix} \right|\\
 &= \left| a_{0,1}(\sigma_{i,1,N}- \sigma_{i,1} ) + \sum\limits_{j = 2}^{r_e} a_{0, j} \sigma_{i,j,N}+ \sum\limits_{j = r_e +1}^{d} a_{0, j} \sigma_{i,j,N+1} + a_{k',j'}v_{N+k',i}^{j'} \right|.
\end{align*}
Now, $|\sigma_{i,1,N}- \sigma_{i,1}| = \sum\limits_{k= N+1}^{+ \infty} \frac{u_{k}^{(i,1)}}{\theta^{\floor{\alpha^k}}} \leq 4 \theta^{-\floor{\alpha^{N+1}}}$ if $N$ is large enough. Thus,
\begin{align*}
 \| X \wedge Y_1 \| &\geq |a_{j',k'}v_{N+k',i}^{j'}| -| a_{0,1}(\sigma_{i,1,N}- \sigma_{i,1} )| -\left| \sum\limits_{j = 2}^{r_e} a_{0, j} \sigma_{i,j,N}\right| - \left| \sum\limits_{j = r_e +1}^{d} a_{0, j} \sigma_{i,j,N+1} \right| \\
 &\geq \frac{1}{qd} - 4 \theta^{-\floor{\alpha^{N+1}}} - d \sigma \left(\frac{\theta^{-\floor{ \alpha^{N+q_e+1}}}}{(\sigma qd)^2}\right) \\
 &\geq \theta^{-\floor{ \alpha^{N+q_e+1}}}
\end{align*}
if $N$ is large enough, which proves $(\ref{8minor_angle_y1_CNe_proof})$. Now, let us suppose that $k' = 0 $ and $j'= 1$, so $|a_{0,1}| \geq \frac{1}{qd}$. We set $ i = N+q_e +1 \mod(qd)$. By definition of $u_{N+k}^{(i,j)}$ in $(\ref{8construc_suite_u})$, we have 
 $\forall k \in \llbracket 1, q_e \rrbracket, \quad \forall j \in \llbracket 2, d \rrbracket, \quad v_{N+k,i}^j = 0$
and $u_{N+1}^{(i,1)} = \ldots = u_{N+q_e}^{(i,1)} = 0$ and $u_{N+q_e+ 1}^{(i,1)} \in \{2,3\}$. 
In particular, $\sigma_{i,1,N}= \sum\limits_{k=0}^{N} \frac{u_{k}^{(i,1)}}{\theta^{\floor{\alpha^k}}} = \sum\limits_{k=0}^{N+q_e} \frac{u_{k}^{(i,1)}}{\theta^{\floor{\alpha^k}}} = \sigma_{i,1,N+q_e}$ and $ \sum\limits_{j = 1}^{r_e} \sum\limits_{k = 1}^{q_e} a_{k,j}v_{N+k,i}^j + \sum\limits_{j = r_e +1}^{d} \sum\limits_{k =1}^{q_e-1} a_{k,j}v_{N+k+1,i}^j = 0.$
By bounding $\| X \wedge Y_1\| $ from below by the minor of $(X| Y_1)$ corresponding to the first row and the $(i+1+d)$-th row, we find:
\begin{align*}
 \| X \wedge Y_1 \| &\geq \left| \det\begin{pmatrix}
 1 & a_{0,1} \\
 \sigma_{i,1} & \sum\limits_{j = 1}^{r_e} a_{0, j} \sigma_{i,j,N}+ \sum\limits_{j = r_e +1}^{d} a_{0, j} \sigma_{i,j,N+1} 
 \end{pmatrix} \right| \\ &= \left| a_{0,1}(\sigma_{i,1,N}- \sigma_{i,1} ) + \sum\limits_{j = 2}^{r_e} a_{0, j} \sigma_{i,j,N}+ \sum\limits_{j = r_e +1}^{d} a_{0, j} \sigma_{i,j,N+1} \right|.
\end{align*}
Now, $|\sigma_{i,1,N}- \sigma_{i,1}| = |\sigma_{i,1,N+q_e} - \sigma_{i,1}| = \sum\limits_{k= N+q_e+1}^{+ \infty} \frac{u_{k}^{(i,1)}}{\theta^{\floor{\alpha^k}}} \geq \frac{u_{N+q_e +1 }^{(i,1)}}{\theta^{\floor{\alpha^{N+q_e+1}}}} \geq 2 \theta^{-\floor{\alpha^{N+q_e+1}}}$. Therefore,
\begin{align*}
 \| X \wedge Y_1 \| &\geq |a_{0,1}(\sigma_{i,1,N}- \sigma_{i,1} )| - \left| \sum\limits_{j = 2}^{r_e} a_{0, j} \sigma_{i,j,N}\right| - \left| \sum\limits_{j = r_e +1}^{d} a_{0, j} \sigma_{i,j,N+1} \right| \\
 &\geq \frac{2 \theta^{-\floor{\alpha^{N+q_e+1}}}}{qd}- d \sigma \left(\frac{\theta^{-\floor{ \alpha^{N+q_e+1}}}}{(\sigma qd)^2}\right) \\
 &\geq \frac{\theta^{-\floor{ \alpha^{N+q_e+1}}}}{qd}
\end{align*}
which proves $(\ref{8minor_angle_y1_CNe_proof})$. Therefore, $(\ref{8minor_angle_y1_CNe_proof})$ holds in all cases, which concludes the proof of the lemma.

\end{proof}

We can now estimate the last angle (corresponding to $j = d) $ between $A$ and $C_{N,e}$.

\begin{lem}\label{8lem_psid_ACNe}
There exist constants $\cons \label{8cons_minor_angle_psid} >0$ and $ \cons \label{8cons_major_angle_psid}>0$ independent of $N$ such that 
\begin{align*}
c_{\ref{8cons_minor_angle_psid}} H(C_{N,e})^{\frac{- \alpha^{q_e +1}}{r_e + (d-r_e)\alpha}} \leq \psi_d(A, C_{N,e}) \leq c_{\ref{8cons_major_angle_psid}} H(C_{N,e})^{\frac{- \alpha^{q_e +1}}{r_e + (d-r_e)\alpha}}.
\end{align*}
\end{lem}

\begin{proof}
Recall that for $j \in \llbracket 1,d \rrbracket$ and $N \in \N, \psi_1(\Span(Y_j), \Span(X_{N+q_e}^j)) \leq c_{\ref{8cons_major_angle_Yj_XNj}} \theta^{-\alpha^{N+q_e+1}}$ according to $(\ref{8major_Yj_XNj})$. By construction of $C_{N,e}$, we have $\Span(X_{N+q_e}^1, \ldots, X_{N+q_e}^d) \subset C_{N,e}$ and thus 
$$ \psi_d(A,C_{N,e}) \leq \psi_d(A, \Span(X_{N+q_e}^1, \ldots, X_{N+q_e}^d))$$
by the corollary of Lemma~12 of \cite{Schmidt}. According to Theorem~$1.2$ of \cite{joseph_spectre}, we have:
\begin{align*}
\psi_d(A, \Span(X_{N+q_e}^1, \ldots, X_{N+q_e}^d)) &\leq c_{\ref{8cons_maj_prop45_appliquee}} \sum\limits_{j=1}^{d} \psi_1(\Span(Y_j), \Span(X_{N+q_e}^j) ) 
\end{align*}
where $\cons \label{8cons_maj_prop45_appliquee} $ depends on $Y_1, \ldots, Y_d $ and $n$.
Hence,
\begin{align*}
\psi_d(A, \Span(X_{N+q_e}^1, \ldots, X_{N+q_e}^d)) &\leq c_{\ref{8cons_maj_prop45_appliquee}}c_{\ref{8cons_major_angle_Yj_XNj}} d \theta^{-{\alpha^{N+q_e +1 }} }\\
&\leq c_{\ref{8cons_maj_prop45_appliquee}}c_{\ref{8cons_major_angle_Yj_XNj}} d c_{\ref{8cons_haut_CN_major}}^{\frac{{\alpha^{N+q_e +1 }}}{r_e{\alpha^N} +(d-r_e){\alpha^{N+1}} }} H(C_{N,e})^{\frac{-{\alpha^{N+q_e +1 }}}{r_e{\alpha^N} +(d-r_e){\alpha^{N+1}} }} \\
&= c_{\ref{8cons_major_angle_psid}} H(C_{N,e})^{\frac{-{\alpha^{q_e +1 }}}{r_e +(d-r_e)\alpha }} 
\end{align*}
since $H(C_{N,e}) \leq c_{\ref{8cons_haut_CN_major}} \theta^{r_e{\alpha^N} +(d-r_e){\alpha^{N+1}}} $ by Lemma~\ref{8lem_hauteur_CN}. This proves the upper bound of the lemma with $c_{\ref{8cons_major_angle_psid}} = c_{\ref{8cons_maj_prop45_appliquee}}c_{\ref{8cons_major_angle_Yj_XNj}} d c_{\ref{8cons_haut_CN_major}}^{\frac{{\alpha^{q_e +1 }}}{r_e +(d-r_e){\alpha} }} $.

\bigskip 

To establish the lower bound, we use the fact that $\psi_d(A,C_{N,e}) \geq 
\psi_1(\Span(Y_1), C_{N,e}) $ according to Lemma~2.3 of \cite{joseph_spectre}. Lemma~\ref{8lem_minoration_Y1angleCNE} then gives 
$\psi_d(A,C_{N,e}) \geq c_{\ref{8cons_minor_angle_y1_CNe}} \theta^{- \alpha^{N+q_e+1}}.$
Since $H(C_{N,e}) \geq c_{\ref{8cons_haut_CN_minor}} \theta^{r_e{\alpha^N} +(d-r_e){\alpha^{N+1}}} $ according to Lemma~\ref{8lem_hauteur_CN}, we have the lower bound 
\begin{align*}
\psi_d(A,C_{N,e}) \geq c_{\ref{8cons_minor_angle_y1_CNe}}c_{\ref{8cons_haut_CN_minor}}^{\frac{{\alpha^{N+q_e +1 }}}{r_e{\alpha^N} +(d-r_e){\alpha^{N+1}} }} H(C_{N,e})^{\frac{{-\alpha^{N+q_e +1 }}}{r_e{\alpha^N} +(d-r_e){\alpha^{N+1}} }} = c_{\ref{8cons_minor_angle_psid}}H(C_{N,e})^{\frac{-{\alpha^{q_e +1 }}}{r_e +(d-r_e)\alpha }} 
\end{align*}
where $c_{\ref{8cons_minor_angle_psid}} = c_{\ref{8cons_minor_angle_y1_CNe}}c_{\ref{8cons_haut_CN_minor}}^{\frac{{\alpha^{q_e +1 }}}{r_e +(d-r_e){\alpha} }} $, proving the lower bound of the lemma.

\end{proof}

We have thus constructed an infinite family of rational subspaces $C_{N,e}$ of dimension $e$ such that 
$\psi_d(A,C_{N,e}) \leq c_{\ref{8cons_major_angle_psid}}H(C_{N,e})^{\frac{-{\alpha^{q_e +1 }}}{r_e +(d-r_e)\alpha }} $
which implies in particular that $\mu_n(A|e)_d \geq \frac{{\alpha^{q_e +1 }}}{r_e +(d-r_e)\alpha }.$

\subsection{Upper bound on the exponent}
We shall now show that the lower bound found in the previous section is optimal, meaning that we will bound $\mu_n(A|e)_d$ from above by $\frac{{\alpha^{q_e +1 }}}{r_e +(d-r_e)\alpha }$. We state a first technical lemma which will be useful in the proof of the upper bound. This lemma actually generalizes the lower bound of Lemma~\ref{8lem_hauteur_CN}.
\\ Recall that for $N$ and $v$ two positive integers, we have defined:
\begin{align*}
B_{N,v} = \Span(X_N^1, X_{N+1}^1,\ldots, X_{N+v-1}^1, X_N^2 \ldots, X_{N+v-1}^2, \ldots, X_N^d,\ldots, X_{N+v-1}^d).
\end{align*}

\begin{lem}\label{8minoration_somme_espace}
Let $N \in \N$, $v \in \llbracket 1,q -1 \rrbracket $, and $r \in \llbracket 0, d-1 \rrbracket $. For $W $ a rational subspace of $\Span(X_N^1, \ldots, X_N^d)$ of dimension $r$, we have:
\begin{align*}
H(B_{N+1,v}\oplus W) \geq c_{\ref{8cons_minoration_somme_espace}} \theta^{ r {\alpha^N} +(d-r) {\alpha^{N+1}}} 
\end{align*}
where $\cons \label{8cons_minoration_somme_espace} >0$ is independent of $N$ and $W$.
\end{lem}

\begin{proof}
We fix $ U_1, \ldots, U_r $ be a $\Z$-basis of $ W \cap \Z^n $. Since $ W \subset \Span(X_N^1, \ldots, X_N^d) $ and these vectors form a $\Z$-basis of $\Span(X_N^1, \ldots, X_N^d) \cap \Z^n$, we can write for $ i \in \llbracket 1, r \rrbracket $:
\begin{align}\label{8baseW}
U_i = \sum\limits_{j = 1}^d a_{i,j} X_N^j
\end{align}
with $ a_{i,j} \in \Z $.
We note that $ B_{N+1,v} \oplus W $ is a direct sum and has dimension $\dim(B_{N+1,v}) + r = vd + r $ according to Lemma~\ref{8lem_dimBN_zbaseBN}.
The same lemma states also that the vectors $ (X_{N+1}^j)_{j \in \llbracket 1,d \rrbracket} \cup (V_k^j)_{j \in \llbracket 1,d \rrbracket, k \in \llbracket N+2, N+v \rrbracket} $ form a $\Z$-basis of $ B_{N+1,v} \cap \Z^n $.
By concatenating this base of $ B_{N+1,v} $ with the chosen basis $(U_1, \ldots, U_r)$ of $ W $, we form a basis of the real vector subspace $ B_{N+1,v} \oplus W $ (but not necessarily of the $\Z$-module $ (B_{N+1,v} \oplus W) \cap \Z^n $). Moreover, since the vectors of this basis have integers coordinates, the formula $(7)$ stated in \cite[section 3]{Schmidt} gives the height of $ B_{N+1,v} \oplus W $:
\begin{align}\label{8hauteur_somme}
H(B_{N+1,v} \oplus W) = \frac{\left\| (\bigwedge_{j =1}^d X_{N+1}^j) \wedge \left( \bigwedge_{k = N+2}^{N+v} (V_k^1 \wedge \ldots \wedge V_k^d) \right) \wedge U_1 \wedge \ldots \wedge U_r \right\| }{N(I)}
\end{align}
where $N( I ) $ is the norm of the ideal $I$ generated by the Grassmannian coordinates associated with this basis, which, as a reminder, are the minors of size $ vd + r $ of the matrix associated with the vectors $ (X_{N+1}^j)_{j \in \llbracket 1,d \rrbracket} \cup (V_k^j)_{j \in \llbracket 1,d \rrbracket, k \in \llbracket N+2, N+v \rrbracket} \cup (U_i)_{i \in \llbracket 1,r \rrbracket} $.
According to (\ref{8baseW}) we have for $ i \in \llbracket 1,r \rrbracket $:
\begin{align*}
U_i = \sum\limits_{j = 1}^d a_{i,j}X_N^j = \sum\limits_{j = 1}^d a_{i,j}\frac{X_{N+1}^j -U_{N+1}^j}{\theta^{\lfloor \alpha^{N+1} \rfloor - \lfloor \alpha^N \rfloor} }
\end{align*}
using also the formula (\ref{8rec_Xn_Un}).
Thus:
\begin{align}
&\| (\bigwedge_{j =1}^d X_{N+1}^j) \wedge \left( \bigwedge_{k = N+2}^{N+v} (V_k^1 \wedge \ldots \wedge V_k^d) \right) \wedge U_1 \wedge \ldots \wedge U_r \| \nonumber \\
&= \left(\frac{1}{\theta^{\lfloor \alpha^{N+1} \rfloor - \lfloor \alpha^N \rfloor}}\right)^r \| (\bigwedge_{j =1}^d X_{N+1}^j) \wedge \left( \bigwedge_{k = N+2}^{N+v} (V_k^1 \wedge \ldots \wedge V_k^d) \right) \wedge \sum\limits_{j = 1}^d a_{1,j}U_{N+1}^j \wedge \ldots \wedge \sum\limits_{j = 1}^d a_{r,j}U_{N+1}^j \| \nonumber \\
&= \theta^{ r \lfloor \alpha^N \rfloor +(d-r) \lfloor \alpha^{N+1} \rfloor} \| (\bigwedge_{j =1}^d Z_{N+1}^j) \wedge \left( \bigwedge_{k = N+2}^{N+v} (V_k^1 \wedge \ldots \wedge V_k^d) \right) \wedge \sum\limits_{j = 1}^d a_{1,j}U_{N+1}^j \wedge \ldots \wedge \sum\limits_{j = 1}^d a_{r,j}U_{N+1}^j \| \label{8minot_prod_ext}
\end{align}
with $ Z_{N+1}^j = \frac{1}{\theta^{\lfloor \alpha^{N+1} \rfloor}} X_{N+1}^j $.
Now the norm of the exterior product that appears in (\ref{8minot_prod_ext}) can be bounded below by the absolute value of any minor of size $ dv + r $ of the matrix $ M $ whose columns are the vectors $ (Z_{N+1}^j)_{j \in \llbracket 1, d \rrbracket} \cup (V_k^j)_{j \in \llbracket 1,d \rrbracket, k \in \llbracket N+2, N+v \rrbracket} \cup (\sum\limits_{j = 1}^d a_{i,j}U_{N+1}^j)_{i \in \llbracket 1,r \rrbracket} $. We have
 $M = \begin{pmatrix}
 \begin{matrix}
 I_d \\
 \Sigma_{N+1}
 \end{matrix} &
 \begin{matrix}
 A
 \end{matrix}
 \end{pmatrix} \in \text{M}_{n,dv+r}(\R)$
where $ \Sigma_{N+1} = (\sigma_{i,j }^{N+1})_{ i \in \llbracket 0,qd-1 \rrbracket, j \in \llbracket 1,d \rrbracket } $ and $ A \in \text{M}_{n,d(v-1)+r}(\R) $ whose columns are the vectors $ (V_k^j)_{j \in \llbracket 1,d \rrbracket, k \in \llbracket N+2, N+v \rrbracket} \cup (\sum\limits_{j = 1}^d a_{i,j}U_{N+1}^j)_{i \in \llbracket 1,r \rrbracket} $.
According to the construction of the vectors $ V_k^j $ and $ U_{N+1}^j $ in (\ref{8rec_Xn_Un}) and since $ a_{i,j} \in \Z $, there exists $ A' \in \text{M}_{n-d, d(v-1) + r }(\Z) $ a matrix with integer coefficients such that:
\begin{align*}
 M = \begin{pmatrix}
 \begin{matrix}
 I_d \\
 \Sigma_{N+1}
 \end{matrix} &
 \begin{matrix}
 A
 \end{matrix}
 \end{pmatrix} = \begin{pmatrix}
 \begin{matrix}
 I_d \\
 \Sigma_{N+1}
 \end{matrix} &
 \begin{matrix}
 0 \\
 A'
 \end{matrix}
 \end{pmatrix}.
\end{align*}
Moreover, the matrix $A'$ has rank $d(v-1) + r$ because $\operatorname{rank}(M) = dv + r$ since the columns of $M$ are vectors from a basis. Thus, we can extract a minor from
$\begin{pmatrix}
 \begin{matrix}
 I_d \\
 \Sigma_{N+1}
 \end{matrix} &
 \begin{matrix}
 0 \\
 A'
 \end{matrix}
\end{pmatrix}$
that is non-zero and integer. In particular, it is bounded below in absolute value by 1, and by referring to (\ref{8minot_prod_ext}) we have:
\begin{align}\label{8derniere_mino_prod}
 \| (\bigwedge\limits_{j =1}^d X_{N+1}^j) \wedge \left(\bigwedge\limits_{k = N+2}^{N+v} (V_k^1 \wedge \ldots \wedge V_k^d) \right) \wedge U_1 \wedge \ldots \wedge U_r \| \geq \theta^{ r \lfloor\alpha^N\rfloor +(d-r) \lfloor\alpha^{N+1}\rfloor}.
\end{align}
According to (\ref{8hauteur_somme}), it remains to show that $N(I)$ is bounded by a constant depending only on $A$. For this, we show that the family $(X_{N+1}^j)_{j \in \llbracket 1,d\rrbracket} \cup (V_k^j)_{j \in \llbracket 1,d\rrbracket, k \in \llbracket N+2, N+v\rrbracket} \cup (U_i)_{i \in \llbracket 1,r \rrbracket}$ forms "almost" a $\Z$-basis of $(B_{N+1,v}\oplus W)\cap \Z^n$. With this aim in mind, let $(u_j)_{j \in \llbracket 1,d\rrbracket} \cup (w_{j,k})_{j \in \llbracket 1,d\rrbracket, k \in \llbracket N+2, N+v\rrbracket} \cup (v_i)_{i \in \llbracket 1,r \rrbracket} \in [0,1]^{r + vd}$ such that:
\begin{align}\label{8Xentier_proof}
 X = \sum\limits_{ j = 1}^d u_j X_{N+1}^j + \sum\limits_{\underset{j \in \llbracket 1,d\rrbracket }{k \in \llbracket N+2, N+v\rrbracket} } w_{j,k} V_k^j + \sum\limits_{ i= 1}^r v_i U_i \in \Z^n.
\end{align}
We decompose $X$ in the $\Z$-basis of $B_{N,v+1} \cap \Z^n$ explicitly described in Lemma~\ref{8lem_dimBN_zbaseBN}.
We have $ U_i = \sum\limits_{j = 1}^d a_{i,j}X_N^j $ according to (\ref{8baseW}), and $X_{N+1}^j = \theta^{\lfloor\alpha^{N+1}\rfloor - \lfloor\alpha^{N}\rfloor}X_N^j + \| U_{N+1}^j\|V_{N+1}^j$ using the formula (\ref{8rec_Xn_Un}).
For $k \in \llbracket N+2, N+v\rrbracket $, the $V_k^j$ do not appear in the decomposition of $U_i$ and $X_{N+1}^j$. By definition of a $\Z$-basis we then have $\forall j \in \llbracket 1,d\rrbracket, \quad \forall k \in \llbracket N+2, N+v\rrbracket, \quad w_{j,k} \in \Z .$
In particular, $ \sum\limits_{\underset{j \in \llbracket 1,d\rrbracket }{k \in \llbracket N+2, N+v\rrbracket} } w_{j,k} V_k^j \in \Z^n $ and the relation (\ref{8Xentier_proof}) then gives
$ \sum\limits_{ j = 1}^d u_j \left(\theta^{\lfloor\alpha^{N+1}\rfloor - \lfloor\alpha^{N}\rfloor}X_N^j + \| U_{N+1}^j\|V_{N+1}^j\right) + \sum\limits_{j = 1}^d \left(\sum\limits_{ i= 1}^r v_i a_{i,j} \right) X_N^j \in \Z^n.$
According to Lemma~\ref{8lem_dimBN_zbaseBN}, the $V_{N+1}^j$ and $X_N^j$ form a $\Z$-basis of $V \cap \Z^n$, where $V$ is the real vector subspace generated by these vectors. So we have for all $j \in \llbracket 1, d\rrbracket$:
\begin{align}
 u_j \| U_{N+1}^j\| \in \Z, \label{8inZ1} \\
 u_j \theta^{\lfloor\alpha^{N+1}\rfloor - \lfloor\alpha^{N}\rfloor} + \sum\limits_{ i= 1}^r v_i a_{i,j} \in \Z.\label{8inZ2}
\end{align}
Since $ \| U_{N+1}^j\| = 2 $ or $ 3 $, the relation (\ref{8inZ1}) gives $ 6u_j \in \Z $ for all $j \in \llbracket 1,d \rrbracket$. By the second relation (\ref{8inZ2}), we then have $ \sum\limits_{ i= 1}^r6 v_i a_{i,j} \in \Z $ for all $j \in \llbracket 1, d\rrbracket $. Finally:
$ \sum\limits_{i = 1 }^r 6 v_i U_i = \sum\limits_{j= 1}^d \left(\sum\limits_{i = 1 }^r 6 v_i a_{i,j} \right) X_N^{j} \in \Z^n$ and as $U_1, \ldots, U_r$ form a $\Z$-basis of $W \cap \Z^n$ we find $ 6v_i \in \Z $ for all $i \in \llbracket 1, r\rrbracket$.
So we finally have:
\begin{align*}
 w_{j,k} \in \Z& \text{ for } (j,k) \in \llbracket 1,d\rrbracket \times \llbracket N+2, N+v\rrbracket,\\
 6u_j \in \Z& \text{ for } j \in \llbracket 1,d\rrbracket,\\
 6 v_i \in \Z& \text{ for } i \in \llbracket 1,r\rrbracket. 
\end{align*}
In particular, this gives $N(I) \leq 6^{d+r} \leq 6^{2d}$.
By combining this with (\ref{8derniere_mino_prod}) in (\ref{8hauteur_somme}) we find
$ H(B_{N+1,v} \oplus W) \geq 6^{-2d}\theta^{ r \lfloor\alpha^N\rfloor +(d-r) \lfloor\alpha^{N+1}\rfloor} \geq c_{\ref{8cons_minoration_somme_espace}} \theta^{ r \alpha^N +(d-r) \alpha^{N+1}}.$
This is the expected result with $c_{\ref{8cons_minoration_somme_espace}} = 6^{-2d}\theta^{-d} $.

\end{proof}

We can now bound from below the $d$-th angle that $A$ forms with any rational subspace of dimension $e$. We first show that any good approximation of $A$ contains a subspace $B_{N+1,q_e}$ for a specific $N \in \N$.

\begin{lem}\label{8lem_inclusion_BN_C}
Let $ \varepsilon >0$ and $C$ be a rational subspace of dimension $e$ such that $ \psi_d(A,C) \leq H(C) ^{- \frac{{\alpha^{q_e +1 }}}{r_e +(d-r_e){\alpha}} - \varepsilon }.$
Assume $H(C)$ is sufficiently large, and let $N \in \N$ be such that:
\begin{align}\label{8choixN}
 \theta^{\alpha^{N+q_e}} \leq H(C) ^{ \frac{\alpha^{q_e +1 }}{r_e +(d-r_e)\alpha}
 + \frac{\varepsilon}{2} -1 } < \theta^{{\alpha^{N+q_e +1}}}.
\end{align}
Then $B_{N+1,q_e} \subset C$.
\end{lem}

\begin{proof}
Let $N$ be the unique integer satisfying $(\ref{8choixN})$. Let $Z_1, \ldots, Z_{e}$ be a $\Z$-basis of $C \cap \Z^n$. For $j \in \llbracket 1,d \rrbracket$, we study 
$\DD_{j,k} = \| X_{N+k}^j \wedge Z_1 \wedge \ldots \wedge Z_{e} \|$ for $k \in \llbracket 1,q_e \rrbracket$. Given a vector $ x \in \R^n $ and a subspace $ V \subset \R^n $, we denote by $ p_V(x) $ the orthogonal projection of $ x $ onto $ V $. One has $
 \DD_{j,k} = \| p_{B^\perp}(X_{N+k}^{j}) \| \| Z_1 \ldots \wedge Z_e \| = \psi_1(\Span(X_{N+k}^j),C) \|X_{N+k}^j \| H(C).$ 
We then have $\DD_{j,k} \leq c_{\ref{8cons_major_norme_XN}} \theta^{\alpha^{N+k}}H(C) \left(\omega(X_{N+k}^j, Y_{j}) + \psi_1(\Span(Y_j),C)\right)$
because $\|X_{N+k}^j\| \leq c_{\ref{8cons_major_norme_XN}} \theta^{\alpha^{N+k}} $ from $(\ref{8norme_XN})$, and $\psi_1(\Span(X_{N+k}^j),C) \leq \omega(X_{N+k}^j, Y_{j}) + \psi_1(\Span(Y_j),C)$ by the triangle inequality (see \cite[section 8]{Schmidt}). 
Now $\psi_1(Y_j,C) \leq \psi_d(A, C)$ by Lemma 2.3 of \cite{joseph_spectre} since $e \geq d$, and $\omega(X_{N+k}^j, Y_{j}) \leq c_{\ref{8cons_major_angle_Yj_XNj}} \theta^{-{\alpha^{N+k+1}}}$ from $(\ref{8major_Yj_XNj})$, thus:
\begin{align*}
\DD_{j,k} &\leq c_{\ref{8cons_maj_1Djk}}\theta^{{\alpha^{N+k}}} H(C) \left(\theta^{-{\alpha^{N+k+1}}} + H(C) ^{\frac{-\alpha^{q_e +1 }}{r_e +(d-r_e)\alpha} - \varepsilon }\right)
\end{align*}
where $\cons \label{8cons_maj_1Djk} >0$ is independent of $N$. 
Furthermore, $\theta^{{\alpha^{N+k}}} \leq \theta^{{\alpha^{N+q_e}}} \leq H(C) ^{ \frac{\alpha^{q_e +1 }}{r_e +(d-r_e)\alpha} + \frac{\varepsilon}{2} -1 } $ from the choice of $N$ in $(\ref{8choixN})$, hence $\DD_{j,k} \leq c_{\ref{8cons_maj_1Djk}} \theta^{{\alpha^{N+k}} -{\alpha^{N+k+1}}} H(C) + c_{\ref{8cons_maj_1Djk}}H(C) ^{-\frac{ \varepsilon }{2}}.$
Moreover, $H(C) ^{ \frac{\alpha^{q_e +1 }}{r_e +(d-r_e)\alpha} + \frac{\varepsilon}{2} -1} \leq \theta^{{\alpha^{N+q_e + 1}}} $ from $(\ref{8choixN})$, hence $\theta \geq H(C)^{ \frac{ \frac{\alpha^{q_e +1 }}{r_e +(d-r_e)\alpha} + \frac{\varepsilon}{2} -1 }{{\alpha^{N+q_e + 1}}} }. $ 
Thus, we obtain the upper bound: 
\begin{align}\label{8major_Djk}
\DD_{j,k} &\leq c_{\ref{8cons_maj_1Djk}} H(C)^{1+ \frac{ ({\alpha^{N+k}} -{\alpha^{N+k+1}})(\frac{{\alpha^{q_e +1 }}}{r_e +(d-r_e)\alpha} + \frac{\varepsilon}{2} -1 )}{{\alpha^{N+q_e + 1}}}} + c_{\ref{8cons_maj_1Djk}}H(C) ^{-\frac{ \varepsilon }{2}}.
\end{align}
We examine the exponent, and since $ k \geq 1 $:
\begin{align*}
1+ \frac{ ({\alpha^{N+k}} -{\alpha^{N+k+1}})(\frac{{\alpha^{q_e +1 }}}{r_e +(d-r_e)\alpha} + \frac{\varepsilon}{2} -1 )}{{\alpha^{N+q_e + 1}}} &\leq 1+ \frac{ ({\alpha^{N+1}} -{\alpha^{N+2}})(\frac{{\alpha^{q_e +1 }}}{r_e +(d-r_e)\alpha} + \frac{\varepsilon}{2} -1 )}{{\alpha^{N+q_e + 1}}} \\
&= 1+ \frac{ (1 -{\alpha})(\frac{{\alpha^{q_e +1 }}}{r_e +(d-r_e)\alpha} + \frac{\varepsilon}{2} -1 )}{{\alpha^{q_e }}}\\
&\leq \frac{\alpha^{q_e } + (1 -{\alpha})(\frac{{\alpha^{q_e }}}{d} + \frac{\varepsilon}{2} -1 )}{{\alpha^{q_e }}}
\end{align*}
because $\frac{{\alpha^{q_e +1 }}}{r_e +(d-r_e)\alpha} \geq \frac{{\alpha^{q_e }}}{d} \geq 1 $ and $1-\alpha <-1$. Moreover,
\begin{align*}
\alpha^{q_e } + (1 -{\alpha})(\frac{{\alpha^{q_e }}}{d} -1 ) \leq \frac{1}{d}\left(-\alpha^{q_e+1} + \alpha^{q_e}(d+1) + d(\alpha - 1) \right) \leq \frac{1}{d}\left(-\alpha^{2} + \alpha(2d+2) - d \right) 
\end{align*}
since $e \geq d$, hence $q_e \geq 1$. By inequality $(\ref{8lem_inegal_alpha1})$, we have $-\alpha^{2} + \alpha(2d+2) - d \leq 0$. Then, $\frac{\alpha^{q_e } + (1 -{\alpha})(\frac{{\alpha^{q_e }}}{d} + \frac{\varepsilon}{2} -1 )}{{\alpha^{q_e }}} \leq \frac{(1-\alpha)\varepsilon}{2 \alpha^{q_e}} \leq \frac{-\varepsilon}{2 \alpha^{q}}$ because $1-\alpha <-1$ and $q_e \leq q$.
Thus, by revisiting (\ref{8major_Djk}): 
\begin{align*}
\DD_{j,k} &\leq c_{\ref{8cons_maj_1Djk}}H(C)^{- \frac{\varepsilon}{2\alpha^q}} + c_{\ref{8cons_maj_1Djk}}H(C) ^{-\frac{ \varepsilon }{2}} \leq 2c_{\ref{8cons_maj_1Djk}}H(C)^{- \frac{\varepsilon}{2\alpha^q}} .
\end{align*}
For $H(C)$ sufficiently large depending on $\varepsilon, \alpha, n $, and $c_{\ref{8cons_maj_1Djk}} $, we have $ \DD_{j,k} < 1$ for all $j \in \llbracket 1,d \rrbracket $ and $k \in \llbracket 1,q_e \rrbracket$. It implies that for such $(j,k)$ the infinite norm $ \| X_{N+k}^j \wedge Z_1 \wedge \ldots \wedge Z_{e} \|_{\infty} $ vanishes since it is an integer bounded from above by $\DD_{j,k} $. Thus, we have a linear relation between the vectors $X_{N+k}^j, Z_1, \ldots, Z_{e}$ that gives $\forall j \in \llbracket 1,d \rrbracket, \quad \forall k \in \llbracket 1,q_e \rrbracket, \quad X_{N+k}^j \in C.$
These vectors generate $B_{N+1,q_e}$, hence $B_{N+1,q_e } \subset C$. 

\end{proof}

\begin{lem}\label{8lem_minoration_exposant}
 Let $ \varepsilon > 0$. For all but a finite number of rational subspaces $C$ of dimension $e$, we have $$\psi_d(A,C) > H(C)^{- \frac{{\alpha^{q_e +1 }}}{r_e +(d-r_e){\alpha}} - \varepsilon }.$$

\end{lem}

\begin{proof}
Assume the contrary. Then, for some $\varepsilon > 0$, there exist infinitely many rational subspaces $C$ of dimension $e$ satisfying:
\begin{align}\label{8hyp_absurde_proof_lem14}
\psi_d(A,C) \leq H(C)^{\frac{-\alpha^{q_e +1 }}{r_e +(d-r_e){\alpha}} - \varepsilon}.
\end{align}
According to Lemma~\ref{8lem_inclusion_BN_C}, if such a $C$ has sufficiently large height, which we can assume this there infinitely many such $C$, then $B_{N+1, q_e} \subset C$ with $N$ satisfying (\ref{8choixN}). We divide the proof into two cases depending on the value of $r_e$.

\textbullet \: \underline{First case $r_e = 0$:} In this case, we have $e = q_e d$ and $\frac{{\alpha^{q_e +1 }}}{r_e +(d-r_e){\alpha}} = \frac{\alpha^{q_e}}{d}$. We already have $B_{N+1,q_e} \subset C$; by equality of dimensions and since $r_e = 0$, we have $C = B_{N+1,q_e} = C_{N,e}$ by Remark~\ref{8req_CN=BN}.
Thus, $H(C) = H(C_{N,e})$ and by Lemma~\ref{8lem_psid_ACNe} we have 
$c_{\ref{8cons_minor_angle_psid}}H(C)^{\frac{-{\alpha^{q_e+1}}}{r_e+(d-r_e){\alpha}} }\leq \psi_d(A,C_{N,e}) = \psi_d(A,C) \leq H(C) ^{-\frac{{\alpha^{q_e +1 }}}{r_e +(d-r_e){\alpha}} - \varepsilon }.$
This implies $\frac{{\alpha^{q_e+1}}}{r_e +(d-r_e){\alpha}} \geq \frac{{\alpha^{q_e +1 }}}{r_e +(d-r_e){\alpha}} + \frac{\varepsilon }{2}$ since we can make $H(C)$ tend to $+\infty$. Thus, we obtain $\frac{\varepsilon}{2} \leq 0$, which is a contradiction.

\bigskip

\textbullet \: \underline{Second case $r_e \neq 0$:} In this case, $e = q_e d + r_e$ and in particular $1 \leq q_e \leq q-1$. Let us prove that $\Span(X_N^1, \ldots, X_N^d) \cap C $ has dimension greater than or equal to $r_e$ by showing by induction on $r \in \llbracket 0,r_e-1 \rrbracket$ that there exist $r+1$ linearly independent integer vectors in $\Span(X_N^1, \ldots, X_N^d) \cap C $.
Let $r \in \llbracket 0, r_e -1 \rrbracket$. Suppose there exist $U_1, \ldots, U_r$ linearly independent integer vectors in $\Span(X_N^1, \ldots, X_N^d) \cap C $; if $r = 0$, this assumption is vacuous and thus true. Let us show that there exists $U_{r+1} \in \Span(X_N^1, \ldots, X_N^d) \cap C \cap \Z^n $ such that $U_1, \ldots, U_{r+1}$ are linearly independent.

As $q_e \leq q-1$ and the family of vectors involved in the definition of $B_{N+1, q_e}$ in (\ref{8def_Bnv}) is free, we have $B_{N+1,q_e} \cap \Span(X_N^1, \ldots, X_N^d) = \{0\}$, and therefore, according to Lemma~\ref{8lem_dimBN_zbaseBN},
$\dim(B_{N+1,q_e} \oplus \Span(U_1, \ldots, U_r)) = dq_e +r.$
We denote $G_r =B_{N+1,q_e} \oplus \Span(U_1, \ldots, U_r) $ and $D_r = G_r^\perp \cap C$ the orthogonal complement of $G_r$ in $C$.
We have $\dim(D_r) = e - dq_e -r = r_e - r \geq 1$. 
Let $\pi_r : C \to D_r $ be the orthogonal projection onto $D_r$.
We define $\Delta_r = \pi_r(C \cap \Z^n)$. Then $\Delta_r$ is a Euclidean lattice of $D_r$ with determinant 
$ d(\Delta_r) = \frac{H(C)}{H(G_r)}$ ;
this result can be found in the proof of Theorem~2 of \cite{Schmidt}. According to Minkowski's theorem (see \cite[Lemma $4B$]{Schmidt_book}), there exists $X_r' \in \Delta_r \setminus \{0 \} \subset D_r \cap \Q^n$ such that:
\begin{align}\label{8normeXr}
 \| X_r' \| \leq c_{\ref{8cons_minko_1}} d(\Delta_r)^{\frac{1}{\dim(D_r)}} \leq c_{\ref{8cons_minko_1}} \left(\frac{H(C)}{H(G_r)}\right)^{\frac{1}{r_e-r}}
\end{align}
with $\cons \label{8cons_minko_1} >0$ a constant depending only on $e$. Since $X_r' \in \Delta_r$, there exists $X_r \in C \cap \Z^n$ such that $\pi_r(X_r) = X_r'$; we have $X_r \notin G_r$ so that $X_r \notin B_{N+1,q_e}.$
We define $E_r$ as the exterior product of the vectors $(X_{N+k}^j)_{j \in \llbracket 1,d\rrbracket,k \in \llbracket 0, q_e \rrbracket} $ and $X_r$. We aim to show that $E_r= 0$ and for this, we examine its norm. 
We have:
\begin{align*}
 \| E_r \| &= \|X_r \wedge \bigwedge\limits_{k=0}^{q_e} (X_{N+k}^1 \wedge \ldots \wedge X_{N+k}^d) \| = \|\pi_r(X_r) \wedge \bigwedge\limits_{k=0}^{q_e} (X_{N+k}^1 \wedge \ldots \wedge X_{N+k}^d) \|
\end{align*}
since $X_r - \pi_r(X_r) \in G_r \subset B_{N,q_e+1}$.
Now $\| \pi_r(X_r) \| = \|X_r' \| \leq c_{\ref{8cons_minko_1}} \left(\frac{H(C)}{H(G_r)}\right)^{\frac{1}{r-r_e }}$ according to (\ref{8normeXr}) and thus $\| E_r\| $ is bounded from above by 
\begin{align*}
 \|\pi_r(X_r) \| \cdot \| \bigwedge\limits_{k=0}^{q_e} (X_{N+k}^1 \wedge \ldots \wedge X_{N+k}^d) \| \leq c_{\ref{8cons_minko_1}} \left(\frac{H(C)}{H(G_r)}\right)^{\frac{1}{r-r_e}} \| (X_{N}^1 \wedge \ldots \wedge X_{N}^d) \wedge \bigwedge\limits_{k=1}^{q_e} (U_{N+k}^1 \wedge \ldots \wedge U_{N+k}^d) \|
\end{align*}
using the formula (\ref{8rec_Xn_Un}) on the $X_N^j$ with $U_{N+k}^j = \begin{pmatrix}
 0 &
 \cdots &
 0 &
 u_{N+k}^{(0,j)}& 
 \cdots & u_{N+k}^{(qd-1,j)}
 \end{pmatrix}^\intercal$. According to the construction of the $u_k^{(i,j)}$ in (\ref{8construc_suite_u}) we have $\| U_k^j \| \leq 3 $ for all $k $ and $j$. This yields:
\begin{align*}
 \| E_r \| &\leq c_{\ref{8cons_minko_1}} \left(\frac{H(C)}{H(G_r)}\right)^{\frac{1}{r_e-r}} \| X_{N}^1 \wedge \ldots \wedge X_{N}^d \| \prod\limits_{k=1}^{q_e} (\|U_{N+k}^1 \| \cdots \| U_{N+k}^d \|) \\&\leq 3^{dq}c_{\ref{8cons_minko_1}} \left(\frac{H(C)}{H(G_r)}\right)^{\frac{1}{r-r_e}} \| X_{N}^1 \wedge \ldots \wedge X_{N}^d \|.
\end{align*}
Now $\Span(X_{N}^1, \ldots, X_{N}^d) = B_{N,1}$ according to the construction in (\ref{8def_Bnv}). According to Lemma~\ref{8lem_dimBN_zbaseBN} $X_{N}^1, \ldots, X_{N}^d$ form a $\Z$-basis of $B_{N,1} \cap \Z^n$ and thus $H(B_{N,1}) = \| X_{N}^1 \wedge \ldots \wedge X_{N}^d \|$. Now since $B_{N,v} = C_{N-1,dv}$ for $v \in \llbracket1, q \rrbracket$, Lemma \ref{8lem_hauteur_CN} gives $H(B_{N,1}) \leq c_{\ref{8cons_haut_CN_major}} \theta^{d{\alpha^N}}$ and thus $\| X_{N}^1 \wedge \ldots \wedge X_{N}^d \| \leq c_{\ref{8cons_haut_CN_major}} \theta^{d{\alpha^N}}$. On the other hand, Lemma~\ref{8minoration_somme_espace} yields $H(G_r) = H(B_{N+1,q_e} \oplus \Span(U_1, \ldots, U_r)) \geq
c_{\ref{8cons_minoration_somme_espace}} \theta^{ r {\alpha^N} +(d-r) {\alpha^{N+1}}} $. Thus we have:
 \begin{align*}
 \| E_r \| \leq 3^{dq}c_{\ref{8cons_minko_1}}c_{\ref{8cons_minoration_somme_espace}}^{\frac{-1}{r_e-r}} c_{\ref{8cons_haut_CN_major}}\theta^{ d{\alpha^N} - \frac{r {\alpha^N} +(d-r) {\alpha^{N+1}}}{r_e-r } } H(C)^{\frac{1}{r_e-r} } = 3^{dq}c_{\ref{8cons_minko_1}}c_{\ref{8cons_minoration_somme_espace}}^{\frac{-1}{r_e-r}} c_{\ref{8cons_haut_CN_major}}\theta^{ \frac{\alpha^N (d(r_e-r) - r - (d-r) {\alpha} )} {r_e-r} } H(C)^{\frac{1}{r_e-r} }.
 \end{align*}
 Now $ d(r_e-r) - r - (d-r) {\alpha} \leq dr_e -(d-r_e) \alpha \leq 0 $ according to the inequality $(\ref{8lem_inegal_alpha4})$, and the choice of $N$ in (\ref{8choixN}) gives $H(C)^{ \frac{{\alpha^{q_e }} }{\alpha^{N+q_e +1 }(d -r_e + \frac{1}{2})} } \leq H(C)^{ \frac{ \frac{\alpha^{q_e +1 }}{r_e +(d-r_e)\alpha } + \frac{\varepsilon}{2} -1 }{\alpha^{N+q_e +1 }} } \leq \theta$
 because $\frac{\alpha^{q_e }}{d -r_e + \frac{1}{2} } \leq \frac{\alpha^{q_e +1 }}{r_e +(d-r_e)\alpha } -1 + \frac{\varepsilon}{2} $ according to the inequality $(\ref{8lem_inegal_alpha5})$. 
 The previous inequality then becomes:
 \begin{align*}
 \| E_r \| \leq 3^{dq}c_{\ref{8cons_minko_1}}c_{\ref{8cons_minoration_somme_espace}}^{\frac{-1}{r_e-r}} c_{\ref{8cons_haut_CN_major}} H(C)^{ \frac{\alpha^N\alpha^{q_e } (d(r_e-r) - r - (d-r) {\alpha} )}{\alpha^{N+q_e +1} (r_e-r)(d -r_e + \frac{1}{2})} + \frac{1}{r_e-r} } = 3^{dq}c_{\ref{8cons_minko_1}}c_{\ref{8cons_minoration_somme_espace}}^{\frac{-1}{r_e-r}} c_{\ref{8cons_haut_CN_major}} H(C)^{ \frac{ (d(r_e-r) - r - (d-r) {\alpha} ) + \alpha(d -r_e + \frac{1}{2}) }{\alpha(r_e-r)(d -r_e + \frac{1}{2})} } .
 \end{align*}
 We study the exponent denoted by $\delta =\frac{ (d(r_e-r) - r - (d-r) {\alpha} ) + \alpha(d -r_e + \frac{1}{2}) }{\alpha(r_e-r)(d -r_e + \frac{1}{2})} $, and we can then bound it from above as follows:
\begin{align*}
\delta = \frac{ 1}{\alpha(r_e-r)(d -r_e + \frac{1}{2})} (- \alpha(r_e-r - \frac{1}{2}) + d(r_e-r) - r ) \leq \frac{ 1}{\alpha(r_e-r)(d -r_e + \frac{1}{2})} (- \frac{\alpha}{2} + d(d-1) )
\end{align*}
because $0 \leq r \leq r_e -1$ and $1 \leq r_e -r \leq d-1 $. Finally, according to inequality $(\ref{8lem_inegal_alpha2})$, we have $- \frac{\alpha}{2} + d(d-1) \leq -1$, thus
$\delta \leq \frac{ -1}{\alpha(r_e-r)(d -r_e + \frac{1}{2})} $
and $\| E_r \| \leq 3^{dq}c_{\ref{8cons_minko_1}}c_{\ref{8cons_minoration_somme_espace}}^{\frac{-1}{r_e-r}} c_{\ref{8cons_haut_CN_major}} H(C)^{ \frac{ -1}{\alpha(r_e-r)(d -r_e + \frac{1}{2}) }}$. In particular, if $H(C)$ is large enough, we have $\| E_r \| < 1$. 
Now, since the vectors considered in the exterior product $E_r = X_r \wedge \bigwedge\limits_{k=0}^{q_e} (X_{N+k}^1 \wedge \ldots \wedge X_{N+k}^d) $ belong to $\Z^n$, this exterior product vanishes. 

\bigskip

This implies the existence of $U_{r+1} \in (B_{N+1,q_e} \oplus \Span(X_r)) \cap \Span(X_N^1, \ldots, X_N^d) \setminus \{ 0\} $. Recall that $X_r \notin B_{N+1,q_e}$ and since the subspaces under consideration are rational, we can take $U_{r+1} \in \Z^n$. As $(B_{N+1,q_e} \oplus \Span(X_r)) \subset C$, we also have $U_{r+1} \in C$.

Let us now show that the vectors $U_1, \ldots, U_{r+1}$ are linearly independent over $\R$. Suppose $ \sum\limits_{k =1}^{r+1} \lambda_k U_k =0 $ is a linear dependency relation. 
We apply $\pi_r$, the orthogonal projection onto $D_r = (B_{N+1,q_e} \oplus \Span(U_1, \ldots, U_r))^\perp \cap C $, yielding:
\begin{align}\label{8indep}
 0= \pi_r\left(\sum\limits_{k =1}^{r+1} \lambda_k U_k\right) &= \sum\limits_{k =1}^{r+1} \lambda_k \pi_r(U_k) = \lambda_{r+1} \pi_r(U_{r+1}).
 \end{align}
Now, $U_{r+1} = Z + \mu X_r$ with $Z \in B_{N+1,q_e} $ and $\mu \in \R$. Also, since $U_{r+1} \in \Span(X_N^1, \ldots, X_N^d) $ and $\Span(X_N^1, \ldots, X_N^d) \cap B_{N+1,q_e} = \{0\}$, we have $\mu \neq 0$. Hence $\pi_r(U_{r+1}) = \mu X_r' \neq 0$, implying $\lambda_{r+1} = 0 $ using (\ref{8indep}).
Finally, we find $\lambda_1 = \ldots = \lambda_r = 0$ using the induction hypothesis that $U_1, \ldots, U_r$ are linearly independent, which concludes the induction.

\bigskip

Thus, we have shown that there exists a rational subspace $W \subset \Span(X_N^1, \ldots, X_N^d) \cap C$ of dimension $r_e$. Since $B_{N+1,q_e} \subset C$ and $\Span(X_N^1, \ldots, X_N^d) \cap B_{N+1,q_e} = \{0\}$, we have $C = B_{N+1,q_e} \oplus W $ by equality of dimensions. Lemma~\ref{8minoration_somme_espace} then gives:
\begin{align}\label{8minor_hautC_rappel}
 H(C) \geq c_{\ref{8cons_minoration_somme_espace}} \theta^{ r_e {\alpha^N} +(d-r_e) {\alpha^{N+1}}}.
\end{align}

Furthermore, $C = B_{N+1,q_e} \oplus W \subset C_{N-1, (q_e+1)d}$, and applying Lemma~\ref{8lem_minoration_Y1angleCNE} with $N' = N-1$ and $e' = q_{e'}d$ where $q_{e'} = q_e +1$, we get:
\begin{align}\label{8minor_rappel_y1_C}
 \psi_1(\Span(Y_1), C_{N-1,(q_e+1)d}) \geq c_{\ref{8cons_minor_angle_y1_CNe}} \theta^{- \alpha^{N'+ q_{e'}}} =
 c_{\ref{8cons_minor_angle_y1_CNe}} \theta^{- \alpha^{N+q_e+1}}
\end{align}

Now, using Lemma 2.3 of \cite{joseph_spectre}, since $Y_1 \in A \setminus \{0\} $ and $\dim(A) =d $, we have:
\begin{align}\label{8minor_angle_C}
 \psi_d(A,C) \geq \psi_d(A,C_{N-1,(q_e+1)d}) \geq \psi_1(\Span(Y_1),C_{N-1,(q_e+1)d}).
\end{align}

Combining inequalities $(\ref{8minor_hautC_rappel})$, $(\ref{8minor_rappel_y1_C})$, and $(\ref{8minor_angle_C})$, we get the existence of a constant $\cons \label{8cons_proof_infinit_derniere_cons} >0 $ independent of $C$, such that $ \psi_d(A,C) \geq c_{\ref{8cons_proof_infinit_derniere_cons}} H(C) ^{ \frac{-\alpha^{N+q_e+1}}{r_e {\alpha^N} +(d-r_e) {\alpha^{N+1}}}} = c_{\ref{8cons_proof_infinit_derniere_cons}} H(C) ^{ \frac{-\alpha^{q_e+1}}{r_e +(d-r_e) {\alpha}}}.$
Recalling the assumption made in $(\ref{8hyp_absurde_proof_lem14})$, we have $c_{\ref{8cons_proof_infinit_derniere_cons}} H(C) ^{ \frac{-\alpha^{q_e+1}}{r_e +(d-r_e) {\alpha}}} \leq H(C) ^{ \frac{-\alpha^{q_e+1}}{r_e +(d-r_e) {\alpha}} - \varepsilon}.$
In particular, for all $C$ with sufficiently large $H(C)$, we have $ c_{\ref{8cons_proof_infinit_derniere_cons}} \leq H(C)^{-\varepsilon }$, and thus $ c_{\ref{8cons_proof_infinit_derniere_cons}} = 0$. This is contradictory and concludes the proof of the lemma.

\end{proof}

Lemma~\ref{8lem_minoration_exposant} yields $\mu_n(A|C)_d \leq \frac{\alpha^{q_e +1 }}{r_e +(d-r_e){\alpha}}.$
Since the other inequality has been proven in section \ref{sect_4otherequality}, this completes the proof of theorem~\ref{8theo_dernier_angle} in the case $d \leq e$.

\section{Computation of the exponent in the case \texorpdfstring{$e < d$}{}}
In this section, we show that $\mu_n(A|e)_e = \frac{\alpha}{e}$ for $e \in \llbracket 1,d -1 \rrbracket $. It is worth noting that in the case where $e = d $, this equality also holds; indeed, we proved Theorem~\ref{8theo_dernier_angle} in this case. Since this case also appears in the proofs of subsequent sections, we will reprove that $\mu_n(A|d)_d = \frac{\alpha}{d}$. Throughout the following, we fix $e \in \llbracket 1, d \rrbracket$.

\subsection{Lower bound on the exponent}
The lower bound on the exponent $\mu_n(A|e)_e$ follows the same ideas as in the case $e \geq d$. Here, we introduce a sequence of subspaces that approximate $A$ very well. For $N \in \N$, we define:
\begin{align}\label{def_DN}
 D_{N,e} = \Span(X_N^1, \ldots, X_N^e)
\end{align}
which is a rational subspace of dimension $e$ since $X_N^i \in \Z^n$ for all $i$.

\begin{lem}\label{8lem_hauteur_DN} The vectors $X_N^1, \ldots, X_N^{e}$ form a $\Z$-basis of $D_{N,e} \cap \Z^n$ and 
\begin{align*}
 c_{\ref{8cons_minor_haut_DN}} \theta^{e {\alpha^N}} \leq H(D_{N,e}) \leq c_{\ref{8cons_major_haut_DN}} \theta^{e {\alpha^N}}
\end{align*}
\end{lem}
with $\cons \label{8cons_minor_haut_DN} >0$ and $\cons \label{8cons_major_haut_DN} >0$ independent of $N$.

\begin{proof}
By Lemma~\ref{8lem_dimBN_zbaseBN}, the vectors $X_N^1, \ldots, X_N^e $ are vectors of a $\Z$-basis of $B_{N,v} \cap \Z^n$ introduced in (\ref{8def_Bnv}). They thus form in particular a $\Z$-basis of $W \cap \Z^n$ where $W$ is the subspace they generate, namely $D_{N,e}$. By definition of the height, we have $H(D_{N,e}) = \| X_N^1 \wedge \ldots \wedge X_N^e\| $.
\\ We recall the notation $Z_{N}^j = \frac{1}{\theta^{\lfloor \alpha^N \rfloor}} X_N^j $ for $i \in \llbracket 1,e \rrbracket$ and we have $Z_N^j \underset{N\to+\infty}{\longrightarrow} Y_j$. In particular, we have $ \| Z_N^1 \wedge \ldots \wedge Z_N^e\| \underset{N\to+\infty} {\longrightarrow} \| Y_1 \wedge \ldots \wedge Y_e \| \neq 0.$ This implies 
\begin{align*}
 \theta^{-e\lfloor \alpha^N \rfloor} \| X_N^1 \wedge \ldots \wedge X_N^e\| = \| Z_N^1 \wedge \ldots \wedge Z_N^e\| \underset{N\to+\infty} {\longrightarrow} \| Y_1 \wedge \ldots \wedge Y_e \| \neq 0.
\end{align*}
Thus, there exist $\cons \label{8cons_proof_haut_DN1}>0$ and $\cons \label{8cons_proof_haut_DN2}>0$ such that for all $N \in \N $,
$ c_{\ref{8cons_proof_haut_DN1}} \theta^{e \lfloor \alpha^N \rfloor} \leq H(D_{N,e}) \leq c_{\ref{8cons_proof_haut_DN2}} \theta^{e \lfloor \alpha^N \rfloor}$ which allows us to conclude since $\theta^{e {\alpha^N} -e } \leq \theta^{e \lfloor \alpha^N \rfloor} \leq \theta^{e \alpha^N}$.

\end{proof}

\begin{lem}\label{8lem_Angle_psie_DN}
There exists a constant $\cons \label{8cons_major_psie_A_DNe}>0$ independent of $N$ such that:
\begin{align*}
 \psi_e(A,D_{N,e}) \leq c_{\ref{8cons_major_psie_A_DNe}} H(D_{N,e})^{-\alpha/e}
\end{align*}
for any $N$ sufficiently large.
\end{lem}

\begin{proof}
We recall that for $j \in \llbracket 1,d \rrbracket$ and $N \in \N$, we have
 $\psi_1(\Span(Y_j), \Span(X_{N}^j)) \leq c_{\ref{8cons_major_angle_Yj_XNj}} \theta^{-\alpha^{N+1}}$
according to $(\ref{8major_Yj_XNj})$.
Since $\Span(Y_1,\ldots, Y_e) \subset A $, it follows that $ \psi_e(A,D_{N,e}) \leq \psi_e(\Span(Y_1,\ldots, Y_e), D_{N,e})$ and according to Lemma 6.1 of \cite{joseph_exposants} we have 
$ \psi_e(\Span(Y_1,\ldots, Y_e), D_{N,e}) \leq c_{\ref{8cons_maj_prop45}} \sum\limits_{j=1}^{e} \psi_1(\Span(Y_j), \Span(X_{N}^j) ) $
with $\cons \label{8cons_maj_prop45} >0 $ depending only on $Y_1, \ldots, Y_e $ and $n$.
Thus,
\begin{align*}
 \psi_e(A, D_{N,e}) \leq c_{\ref{8cons_maj_prop45}} c_{\ref{8cons_major_angle_Yj_XNj}} e \theta^{-{\alpha^{N +1 }} }\leq c_{\ref{8cons_maj_prop45}} c_{\ref{8cons_major_angle_Yj_XNj}}d c_{\ref{8cons_major_haut_DN}}^{\frac{{\alpha^{N+1 }}}{e{\alpha^N} }} H(D_{N,e})^{\frac{-{\alpha^{N +1 }}}{e{\alpha^N} }} 
= c_{\ref{8cons_major_psie_A_DNe}} H(D_{N,e})^{-{\alpha}/{e }} 
 \end{align*}
 where $c_{\ref{8cons_major_psie_A_DNe}} = c_{\ref{8cons_maj_prop45}} c_{\ref{8cons_major_angle_Yj_XNj}}d c_{\ref{8cons_major_haut_DN}}^{{\alpha}/{e}} $ and because $H(D_{N,e}) \leq c_{\ref{8cons_major_haut_DN}} \theta^{e {\alpha^N}} $ according to Lemma~\ref{8lem_hauteur_DN}.
 
\end{proof}

Therefore, we have constructed an infinite sequence of rational subspaces $D_{N,e}$ of dimension $e$ such that 
$ \psi_e(A,D_{N,e}) \leq c_{\ref{8cons_major_psie_A_DNe}} H(D_{N,e})^{-\alpha/e}$ ;
this implies in particular $ \mu_n(A|e)_e \geq \frac{\alpha}{e}.$

\subsection{Upper bound of the exponent}
The aim of this section is to establish an upper bound on $\mu_n(A|e)_e$. We first give a necessary condition on the subspaces of best approximation of $A$.

\begin{lem}\label{8lem_intersection_DN_C}
 Let $\varepsilon > 0$ and $C$ be a rational subspace of dimension $e$ such that $\psi_e(A,C) \leq H(C)^{-\frac{\alpha}{e} - \varepsilon}.$
 Then, if $H(C)$ is large enough, there exist $N \in \N$ and $Z \in \Z^n \setminus \{0\}$ such that $Z \in C \cap D_{N,d}$ and 
 \begin{align*}
 \| Z \| \leq c_{\ref{8cons_enonce_lem_minko}}H(C)^{{1}/{e}}
 \end{align*}
 with $\cons \label{8cons_enonce_lem_minko} >0$ independent of $Z$ and $N$.
\end{lem}

\begin{proof}
According to Minkowski's theorem (see \cite[Lemma $4B$]{Schmidt_book}), there exists $Z \in C \cap \Z^n \setminus \{0\}$ such that:
\begin{align}\label{8vecteur_Z}
 \|Z\| \leq c_{\ref{8cons_minko2}} H(C)^{{1}/{e}}
\end{align}
with $\cons \label{8cons_minko2} >0$ a constant independent of $Z$. It remains to show that there exists some $N$ such that $Z \in D_{N,d} = \Span(X_N^1, \ldots, X_N^d)$. Let $Z^A$ be the orthogonal projection of $Z$ onto $A$. We introduce $a_1, \ldots, a_d \in \R$ such that
$ Z^A = \sum\limits_{j = 1}^d a_j Y_j.$
We seek to show that there exists $N$ such that $\|Z \wedge X_N^1 \wedge \ldots \wedge X_N^d\|$ vanishes. Since $X_N^1, \ldots, X_N^d$ form a $\Z$-basis of $D_{N,d} \cap \Z^n$ by Lemma~\ref{8lem_hauteur_DN}, we have, for $N \in \N$:
\begin{align}
 \|Z \wedge X_N^1 \wedge \ldots \wedge X_N^d\| &= \psi_1(\Span(Z), D_{N,d}) H(D_{N,d}) \|Z\| \nonumber\\
 & \leq \omega(Z, \sum\limits_{j = 1}^d a_j X_N^j) H(D_{N,d}) \|Z\| \nonumber\\
 & \leq \left(\omega(Z, Z^A) + \omega(Z^A, \sum\limits_{j = 1}^d a_j X_N^j) \right) H(D_{N,d}) \|Z\| \label{8angle3}.
\end{align}
Recalling the notation $Z_N^j = \theta^{-\floor{\alpha^N}} X_N^j$, we have:
\begin{align*}
 \omega(Z^A, \sum\limits_{j = 1}^d a_j X_N^j) = \omega(Z^A, \sum\limits_{j = 1}^d a_j Z_N^j) \leq \frac{\left\|Z^A - \sum\limits_{j = 1}^d a_j Z_N^j \right\|}{\|Z^A\|} 
 \leq \left\| \sum\limits_{j = 1}^d a_j Y_j \right\|^{-1} \sum\limits_{j = 1}^d |a_j| \left\|Y_j - Z_N^j \right\|.
\end{align*}
By construction of the $Y_j$, we have $\left\| \sum\limits_{j = 1}^d a_j Y_j \right\| \geq \sqrt{\sum\limits_{j=1}^d |a_j|^2} \geq \sum\limits_{j = 1}^d \frac{|a_j|}{d}$. Indeed, for $i \in \llbracket 1,d \rrbracket$, the $i$-th coordinate of $Y_j$ is equal to $1$ if $i = j$ and $0$ otherwise. Furthermore, we have $\|Y_j - Z_N^j \| \leq \|Y_j\| c_{\ref{8cons_major_angle_Yj_XNj}} \theta^{-{\alpha^{N+1}}}$ by $(\ref{8major_Yj_XNj})$. Therefore, we have:
\begin{align}\label{8angle_ZA_DN}
 \omega(Z^A, \sum\limits_{j=1}^d a_j X_N^j) \leq c_{\ref{8cons_angle_ZA_DN}} \theta^{-{\alpha^{N+1}}}
\end{align}
with $\cons \label{8cons_angle_ZA_DN} = {c_{\ref{8cons_major_angle_Yj_XNj}}}{d} \max\limits_{j \in \llbracket 1,d \rrbracket} \|Y_j\| >0$.
On the other hand, using Lemma 2.3 of \cite{joseph_spectre} since $Y_1 \in A \setminus \{0\}$ and $\dim(C) = e$, we have 
 $\omega(Z, Z^A) = \psi_1(\Span(Z), A) \leq \psi_e(C, A).$
The hypothesis of the lemma then gives:
\begin{align}\label{8angle_ZA_C}
 \omega(Z, Z^A) \leq H(C)^{ -\frac{\alpha}{e} - \varepsilon}.
\end{align}
Combining (\ref{8angle_ZA_DN}) and (\ref{8angle_ZA_C}) with (\ref{8angle3}), we get:
$ \| Z \wedge X_N^1 \wedge \ldots \wedge X_N^d \| \leq \left( H(C)^{- \frac{\alpha}{e} - \varepsilon} + c_{\ref{8cons_angle_ZA_DN}} \theta^{- \alpha^{N+1}} \right) H(D_{N,d}) \| Z \|.$
Additionally, by Lemma~\ref{8lem_hauteur_DN}, we know that $H(D_{N,d}) \leq c_{\ref{8cons_major_haut_DN}} \theta^{d \alpha^N}$ and by (\ref{8vecteur_Z}), we have $\| Z \| \leq c_{\ref{8cons_minko2}} H(C)^{\frac{1}{e}}$. Thus, for $N \in \N$, we get:
\begin{align*}
 \| Z \wedge X_N^1 \wedge \ldots \wedge X_N^d \| \leq \left( H(C)^{- \frac{\alpha}{e} - \varepsilon} + c_{\ref{8cons_angle_ZA_DN}} \theta^{- \alpha^{N+1}} \right) c_{\ref{8cons_major_haut_DN}} \theta^{d \alpha^N} c_{\ref{8cons_minko2}} H(C)^{\frac{1}{e}} \leq c_{\ref{8cons_major_final_ext_preuve}} \left( H(C)^{- \frac{\alpha - 1}{e} - \varepsilon} \theta^{d \alpha^N} + \theta^{\alpha^N(d - \alpha)} H(C)^{\frac{1}{e}} \right)
\end{align*}
with $\cons \label{8cons_major_final_ext_preuve} = c_{\ref{8cons_major_haut_DN}} c_{\ref{8cons_minko2}} \max(1, c_{\ref{8cons_angle_ZA_DN}}) >0$ independent of $C$ and $N$. We now choose $N$ as the integer such that:
\begin{align}\label{8defN}
 \theta^{d \alpha^N} \leq H(C)^{\frac{\alpha - 1}{e} + \frac{\varepsilon}{2}} < \theta^{d \alpha^{N+1}}.
\end{align}
This gives $\theta^{d \alpha^N} \leq H(C)^{\frac{\alpha - 1}{e} + \frac{\varepsilon}{2}}$ and $\theta > H(C)^{\frac{\frac{\alpha - 1}{e} + \frac{\varepsilon}{2}}{d \alpha^{N+1}}}$ hence, since $\alpha > d$:
\begin{align}
 \| Z \wedge X_N^1 \wedge \ldots \wedge X_N^d \| &\leq c_{\ref{8cons_major_final_ext_preuve}} \left( H(C)^{- \frac{\varepsilon}{2}} + H(C)^{\alpha^N (d - \alpha) \left( \frac{\frac{\alpha - 1}{e} + \frac{\varepsilon}{2}}{d \alpha^{N+1}} \right) + \frac{1}{e}} \right) \nonumber \\
 &\leq c_{\ref{8cons_major_final_ext_preuve}} \left( H(C)^{- \frac{\varepsilon}{2}} + H(C)^{\frac{(d - \alpha)(2 \alpha - 2 + e \varepsilon) + 2 d \alpha}{2 e d \alpha}} \right) \label{8prod_exterieur_Z_2HC}.
\end{align}
We now study the exponent of the second term:
$ \frac{(d-\alpha)(2\alpha - 2 + e\varepsilon) + 2d \alpha}{2ed\alpha} = \frac{1}{ed\alpha}\left(-\alpha^2 + (1 + 2d)\alpha - d \right) - \frac{\alpha - d}{2d\alpha}\varepsilon.$ \\
According to inequality (\ref{8lem_inegal_alpha3}), we have $-\alpha^2 + (1 + 2d)\alpha - d \leq 0$, and thus
 $\frac{(d-\alpha)(2\alpha - 2 + e\varepsilon) + 2d\alpha}{2ed\alpha} \leq - c_{\ref{8cons_major_epsilon}}\varepsilon$
with $\cons \label{8cons_major_epsilon} = \frac{\alpha - d}{2d\alpha} > 0$ since $\alpha > d$.
Thus, inequality (\ref{8prod_exterieur_Z_2HC}) becomes
 $\| Z \wedge X_N^1 \wedge \ldots \wedge X_N^d \| \leq c_{\ref{8cons_major_final_ext_preuve}} \left(H(C)^{- \frac{\varepsilon}{2}} + H(C)^{- c_{\ref{8cons_major_epsilon}}\varepsilon} \right)$
with $c_{\ref{8cons_major_final_ext_preuve}}$ and $c_{\ref{8cons_major_epsilon}}$ independent of $C$. For sufficiently large $H(C)$, we have $\| Z \wedge X_N^1 \wedge \ldots \wedge X_N^d \| _{\infty} \leq \| Z \wedge X_N^1 \wedge \ldots \wedge X_N^d \| < 1$ and then
\begin{align*}
 Z \in \Span(X_N^1, \ldots, X_N^d) = D_{N,d}
\end{align*}
for $N$ satisfying (\ref{8defN}), which completes the proof.

\end{proof}

We can now prove the result that allows us to provide a lower bound for the exponent.
\begin{lem}\label{8lem_minoration_exposant_e<d}
 Let $ \varepsilon > 0$. For all but a finite number of rational subspaces $C$ of dimension $e$, we have:
 \begin{align*}
 \psi_e(A,C) > H(C)^{- \frac{\alpha}{e} - \varepsilon }.
 \end{align*}
\end{lem}

\begin{proof}
We argue by contradiction and assume that there exists infinitely many rational subspaces $C$ of dimension $e$ such that: 
\begin{align}\label{8hyp_absurd}
 \psi_e(A,C) \leq H(C)^{- \frac{\alpha}{e} - \varepsilon }.
\end{align}

Let $C$ be such a subspace. Lemma \ref{8lem_intersection_DN_C} then gives $N \in \N$ and $Z \in \Z^n \setminus \{0\}$ such that $Z \in C \cap D_{N,d}$ with
 $\| Z \| \leq c_{\ref{8cons_enonce_lem_minko}}H(C)^{\frac{1}{e}}.$ 
Since $Z \in D_{N,d} \cap \Z^n$, we can write $ Z = \sum\limits_{j = 1}^d v_j X_N^j$
 with $v_j \in \Z$ because the $X_N^j$ form a $\Z$-basis of $D_{N,d} \cap \Z^n$ according to Lemma \ref{8lem_hauteur_DN}.
For $j \in \llbracket 1, d \rrbracket$, let $z_j = \theta^{\floor{\alpha^N}} v_j$, which gives $Z = \sum\limits_{j = 1}^d z_j Z_N^j$. 
We now bound $\psi_1(Z,A)$ from below. Recall the notation $Z^A$ for the orthogonal projection of $Z$ onto $A$. We introduce $a_1, \ldots, a_d \in \R$ such that $Z^A = \sum\limits_{j = 1}^d a_j Y_j.$ We also define
$ \Delta = Z^A - \sum\limits_{j = 1}^d z_j Y_j,$
and $\omega = \| Z^A - Z \|.$
Since the $Z_N^j$ and $Y_j$ have their $i$-th coordinate equal to $1$ if $i = j$ and $0$ otherwise, this gives:
\begin{align*}
 Z^A - Z = \begin{pmatrix} a_1 - z_1 & \cdots & a_d - z_d & \star & \ldots & \star \end{pmatrix}^\intercal.
\end{align*}
We then have 
$ \omega^2 \geq \sum\limits_{j = 1}^d (a_j - z_j)^2,$
and thus for all $j \in \llbracket 1, d \rrbracket$, $| a_j - z_j | \leq \omega$. 
This gives us in particular:
\begin{align}\label{8major_delta}
 \| \Delta \| = \left\| \sum\limits_{j = 1}^d (a_j - z_j) Y_j \right\| \leq c_{\ref{8cons_maj_Delta_par_omega}} \omega
\end{align}
with $\cons \label{8cons_maj_Delta_par_omega} = d \max\limits_{j \in \llbracket 1, d \rrbracket} \| Y_j \| >0$.
We can compute $\|Z \wedge Z^A \| =\|Z \wedge \left( \sum\limits_{j = 1 }^d z_j Y_j + Z^A - \sum\limits_{j = 1 }^d z_j Y_j \right) \| $ and then bound it from below:
\begin{align}
\|Z \wedge Z^A \| = \| Z \wedge \sum\limits_{j = 1 }^d z_j Y_j + Z \wedge \Delta \| \geq \| Z \wedge \sum\limits_{j = 1 }^d z_j Y_j \| - \| Z \wedge \Delta \|. \label{8minoration_1_angleZ}
\end{align}
On the one hand, for all $j_0 \in \llbracket 1,d \rrbracket$ and $i \in \llbracket 0, qd-1 \rrbracket$, recalling the definitions of $Y_j$ and $\sigma_{i,j}$ given in section \ref{section_construction} and considering the $j_0$-th and $(d + i)$-th rows of the matrix $\begin{pmatrix}
 Z & \left| \: \sum\limits_{j= 1 }^d z_j Y_j \right.
\end{pmatrix}$, we have $\| Z \wedge \sum\limits_{j= 1 }^d z_j Y_j \| = \| \sum\limits_{j = 1 }^d z_j Z_N^j \wedge \sum\limits_{j = 1 }^d z_j Y_j \| \geq \left| \det \begin{pmatrix}
 z_{j_0} & z_{j_0} \\
 \sum\limits_{j= 1 }^d z_j \sigma_{i,j}^N & \sum\limits_{j = 1 }^d z_j \sigma_{i,j} 
 \end{pmatrix} \right|$ and thus :
\begin{align*}
\| Z \wedge \sum\limits_{j= 1 }^d z_j Y_j \| \geq |z_{j_0}| \left| \sum\limits_{j = 1 }^d z_j (\sigma_{i,j} - \sigma_{i,j}^N) \right| = |z_{j_0}| \left| \sum\limits_{j = 1 }^d z_j \sum\limits_{k = N+1 }^{+\infty} \frac{u_k^{(i,j)}}{\theta^{\floor{\alpha^k}}} \right|. 
\end{align*}
We choose $j_0$ such that $|z_{j_0}| = \max\limits_{j=1}^d{|z_j|} \neq 0$; such a $j_0$ exists because $Z \neq 0$. 
In particular, we have $|z_{j_0}| = \theta^{\floor{\alpha^N}} |v_{j_0}| \geq \theta^{\floor{\alpha^N}}$ since $v_{j_0} \in \Z$.
We choose $i$ such that $u_{N+1}^{(i,j_0)} \neq 0$, that is $i = \phi_{j_0}(N+1)$ using the notations from section \ref{section_construction}. Then, we have $u_{N+1}^{(i,j_0)} \geq 2$ and $u_{N+1}^{(i,j)} = 0$ for all $j \neq j_0$, hence:
\begin{align*}
 \| Z \wedge \sum\limits_{j= 1 }^d z_j Y_j \| &\geq |z_{j_0}| \left| \sum\limits_{j = 1 }^d z_j \sum\limits_{k = N+1 }^{+\infty} \frac{u_k^{(i,j)}}{\theta^{\floor{\alpha^k}}} \right| \\ 
 &\geq |z_{j_0}| \left| z_{j_0} \frac{u_{N+1}^{(i,j_0)}}{\theta^{\floor{\alpha^{N+1}}}} + \sum\limits_{j = 1 }^d z_j \sum\limits_{k = N+2 }^{+\infty} \frac{u_k^{(i,j)}}{\theta^{\floor{\alpha^k}}} \right| \\
 &\geq |z_{j_0}| \left(|z_{j_0}| \frac{2}{\theta^{\floor{\alpha^{N+1}}}} - d|z_{j_0}| \sum\limits_{k = N+2 }^{+\infty} \frac{u_k^{(i,j)}}{\theta^{\floor{\alpha^k}}} \right) \\
 &\geq |z_{j_0}|^2 \frac{1}{\theta^{\floor{\alpha^{N+1}}}}
\end{align*}
since if $N$ is large enough, we have $d \sum\limits_{k = N+2 }^{+\infty} \frac{u_k^{(i,j)}}{\theta^{\floor{\alpha^k}}} \leq \frac{1}{\theta^{\floor{\alpha^{N+1}}}}$.
We deduce that:
\begin{align}\label{8minorer_prod_ZYj}
 \| Z \wedge \sum\limits_{j= 1 }^d z_j Y_j \| &\geq \frac{\theta^{2\floor{\alpha^{N}}}}{\theta^{\floor{\alpha^{N+1}}}}.
\end{align}
On the other hand, according to (\ref{8major_delta}), we have:
\begin{align}\label{8majorer_prod_Zomega}
 \| Z \wedge \Delta \| \leq \|Z\| \|\Delta \| \leq c_{\ref{8cons_maj_Delta_par_omega}} \omega \|Z \|.
\end{align} 
Combining (\ref{8minoration_1_angleZ}) with (\ref{8minorer_prod_ZYj}) and (\ref{8majorer_prod_Zomega}), we find $ \|Z \wedge Z^A \| \geq \frac{\theta^{2\floor{\alpha^{N}}}}{\theta^{\floor{\alpha^{N+1}}}} - c_{\ref{8cons_maj_Delta_par_omega}} \omega \|Z\|.$
We have $\omega = \| Z^A - Z \| = \| Z \| \omega(Z^A,Z)$ and thus, since $\|Z^A \| \leq \|Z \|$, we get:
\begin{align*}
 \omega = \| Z \| \frac{\| Z^A \wedge Z \|}{\| Z \| \cdot \|Z^A \|} 
 \geq \frac{\| Z^A \wedge Z \|}{\| Z \| } \geq \frac{\theta^{2\floor{\alpha^{N}}}}{\|Z\|\theta^{\floor{\alpha^{N+1}}}} - c_{\ref{8cons_maj_Delta_par_omega}} \omega .
\end{align*}
Finally, we have
 $\omega \geq \frac{c_{\ref{8cons_minor_omega}}}{\|Z \|\theta^{\floor{\alpha ^{N+1}} -2 \floor{\alpha^N}}}$
with $\cons \label{8cons_minor_omega} = (1 +c_{\ref{8cons_maj_Delta_par_omega}})^{-1} >0 $.
We also have $\omega(Z, Z^A)=\frac{\omega}{\|Z\|} \geq \frac{c_{\ref{8cons_minor_omega}}}{\|Z \|^2\theta^{\floor{\alpha ^{N+1}} - 2 \floor{\alpha^N}}},$
where $c_{\ref{8cons_minor_omega}} > 0$ is independent of $N$ and $Z$. 
Recalling that $Z \in C$, and thus $ \omega(Z, Z^A) = \psi_1(\Span(Z),A) \leq \psi_e(C,A)$, by (\ref{8hyp_absurd}) we have:
\begin{align}\label{8majorationderniereZ}
 \frac{c_{\ref{8cons_minor_omega}}}{\|Z \|^2\theta^{\floor{\alpha ^{N+1}} - 2 \floor{\alpha^N}}} \leq H(C)^{-\frac{\alpha}{e} - \varepsilon}.
\end{align}
Moreover, we have chosen $Z$ such that $\|Z\| \leq c_{\ref{8cons_enonce_lem_minko}}H(C)^{{1}/{e}}$. Hence, inequality (\ref{8majorationderniereZ}) becomes\\
$ \frac{c_{\ref{8cons_minor_omega}}}{c_{\ref{8cons_enonce_lem_minko}}^{\alpha+ e\varepsilon }\|Z \|^2\theta^{\floor{\alpha ^{N+1}} - 2 \floor{\alpha^N}}} \leq \|Z\|^{-{\alpha} - e\varepsilon}.$
In particular, we have:
\begin{align}\label{8derniere_contradiction}
 \frac{c_{\ref{8cons_derniere_contradiction}}}{\|Z \|^2\theta^{\floor{\alpha ^{N+1}} - 2 \floor{\alpha^N}}} \leq \|Z\|^{-{\alpha} - e\varepsilon} 
\end{align}
with $ \cons \label{8cons_derniere_contradiction} = c_{\ref{8cons_minor_omega}}c_{\ref{8cons_enonce_lem_minko}}^{-\alpha-e\varepsilon } >0$.
On the other hand, by the construction of $Z_N^j$, we have:
 $\forall j \in \llbracket 1,d \rrbracket, \quad \|Z\| \geq |z_j|.$
Indeed, for $i \in \llbracket 1,d \rrbracket$, the $i\tir$th coordinate of $Z_N^j$ is equal to $1$ if $i =j$ and $0$ otherwise. In particular, we have $\|Z \| \geq |z_{j_0}| \geq \theta^{\floor{\alpha ^N}}$ since $z_{j_0} \neq 0$. Combining this with (\ref{8derniere_contradiction}), we find $ c_{\ref{8cons_derniere_contradiction}} \leq \theta^{-\floor{\alpha ^N}(\alpha + e\varepsilon - 2) + \floor{\alpha ^{N+1}} - 2 \floor{\alpha^N}}$
 with, as a reminder, $c_{\ref{8cons_derniere_contradiction}} >0 $ being a constant independent of $N$. We have:
\begin{align*}
-\floor{\alpha ^N}(\alpha + e\varepsilon - 2) + \floor{\alpha ^{N+1}} - 2 \floor{\alpha^N} = \floor{\alpha ^{N+1}} -\floor{\alpha ^N}(\alpha + e\varepsilon) 
\underset{N \to + \infty}{\longrightarrow} - \infty.
\end{align*}
If $H(C)$ is large, then $N$ is also large by (\ref{8choixN}). Therefore, as $H(C)$ tends to $+\infty$, we find $c_{\ref{8cons_derniere_contradiction}} = 0$, which is contradictory and completes the proof of the lemma.

\end{proof}

Lemma \ref{8lem_minoration_exposant_e<d} then yields for $e \in \llbracket 1,d-1 \rrbracket $,
 $\mu_n(A|C)_e \leq \frac{\alpha}{e}.$
Thus, we have proven Theorem \ref{8theo_dernier_angle} in the case where $ e < d$.

\bigskip

\textbf{Acknowledgements:} I warmly thank Stéphane Fischler for his reading and comments of this paper.

\bibliographystyle{smfplain.bst}
\bibliography{bibliographie.bib}

Gaétan Guillot, Université Paris-Saclay, CNRS, Laboratoire de mathématiques d’Orsay, 91405 Orsay, France.
\\ {Email adress:} guillotgaetan1@gmail.com 
\bigskip

MSC: 11J13, 11H25
\end{document}